\documentclass[reqno]{amsart}

\usepackage{algorithm}
\usepackage{algpseudocode}

\usepackage{amsmath,bm}
\usepackage{amsthm}
\usepackage{amsfonts}
\usepackage{amssymb}
\usepackage{verbatim}
\usepackage{latexsym}
\usepackage[mathscr]{eucal}
\usepackage{color}
\usepackage{enumerate}
\usepackage[active]{srcltx}

\usepackage[colorlinks=true]{hyperref}
\hypersetup{urlcolor=blue, citecolor=red}

\newcommand{\D}{\displaystyle }
\newcommand{\mbb}[1]{\mathbb{#1}}
\newcommand{\R}{\mathbb R}
\newcommand{\Z}{\mathbb Z}

\newcommand{\sH}{\mbb H}
\newcommand{\mbf}[1]{\boldsymbol{#1}}
\newcommand{\bk}{{\mbf k}}
\newcommand{\bh}{{\mbf h}}
\newcommand{\bl}{{\mbf l}}
\newcommand{\bx}{{\mbf x}}

\newcommand{\bu}{{\mbf u}}
\newcommand{\bv}{{\mbf v}}
\newcommand{\bw}{{\mbf w}}
\newcommand{\bphi}{{\mbf \phi}}
\newcommand{\bpsi}{{\mbf \psi}}
\newcommand{\bbeta}{{\mbf \beta}}
\newcommand{\f}{\mbf f}
\newcommand{\g}{\mbf g}
\newcommand{\mV}{\mathcal{V}}
\newcommand{\rd}{{\text{\rm d}}}
\newcommand{\bs}{\boldsymbol}
\newcommand{\per}{\textnormal{per}}
\newcommand{\be}{\begin{equation}}
\newcommand{\ee}{\end{equation}}
\newcommand{\cal}[1]{\mathcal#1}
\newcommand{\co}{\cal O}
\newcommand{\lra}{\longrightarrow}
\newcommand{\ra}{\rightarrow}

\newcommand{\mA}{\mathcal{A}}
\newcommand{\Lloc}{L_{\textnormal{loc}}}
\newcommand{\dt}[1]{{\D\frac{\text{d}#1}{\text{dt}}}}
\newcommand{\mB}{\mathcal{B}}

\newtheorem{theorem}{Theorem}[section]
\newtheorem{corollary}[theorem]{Corollary}
\newtheorem{lemma}[theorem]{Lemma}

\newtheorem{fact}[theorem]{Fact}

\theoremstyle{definition}
\newtheorem{definition}[theorem]{Definition}
\newtheorem{system}[theorem]{System}
\newtheorem{example}[theorem]{Example}
\newtheorem{problem}[theorem]{Problem}

\theoremstyle{remark}
\newtheorem{remark}[theorem]{Remark}

\numberwithin{equation}{section}

\newcommand{\comments}[1]{ {\color{blue} #1}}
\newcommand{\optionalproof}[1]{{\color{red} \begin{proof} {#1} \end{proof}}}
\newcommand{\verbose}[1]{{\color{blue} #1}}
\renewcommand{\comments}[1]{}
\renewcommand{\optionalproof}[1]{}
\renewcommand{\verbose}[1]{}

\newcommand{\W}{\mathbb{W}}
\newcommand{\Wmap}[2]{\W(#1, #2)}
\newcommand{\Proj}[2][N]{P_{#1} #2}

\newcommand{\Cost}{\mathcal{L}}
\newcommand{\A}{A}
\newcommand{\B}[2]{B\left(#1, #2\right)}
\renewcommand{\b}[3]{b\left(#1, #2, #3\right)}
\newcommand{\NSE}[2]{\frac{d}{dt} #1 + #2 \A #1 + \B{#1}{#1}}
\newcommand{\F}[1]{\widehat{#1}}
\newcommand{\conv}[2]{i\kappa_0 P_\sigma \sum_{\bh} (\bk \cdot \F{#1}(\bh)) \F{#2}(\bk-\bh)}
\newcommand{\avg}[1]{\left< #1 \right>_s^t}
\newcommand{\NSEFourier}[3][\conv]{\frac{d}{dt} \F{#2}(\bk) + #3 \kappa_0^2 |\bk|^2 \F{#2}(\bk) + #1{#2}{#2}}

\title[Determining viscosity from modal observations]{Determining the viscosity of the Navier-Stokes equations from observations of finitely many modes}

\author{Animikh Biswas \and Joshua Hudson}

\begin{document} 


\begin{abstract}
	In this work, we ask and answer the question: when is the viscosity of a fluid uniquely determined from spatially sparse measurements of its velocity field? We pose the question mathematically as an optimization problem using the determining map (the mapping of data to an approximation made via a nudging algorithm) to define a loss functional, the minimization of which solves the inverse problem of identifying the true viscosity given the measurement data. We give explicit \emph{a priori} conditions for the well-posedness of this inverse problem. In addition, we show that smallness of the loss functional implies proximity to the true viscosity. We then present an algorithm for solving the inverse problem and provide \emph{a priori} verifiable conditions that ensure its convergence.
\end{abstract}


\maketitle

\section{Introduction}
In many instances, the general form of a particular dynamical system is commonly derived from well-accepted fundamental physical, biological or epidemiological principles. However, frequently these systems have one or multiple parameters which play a crucial role in their dynamical evolution. For instance, in  geophysical models such as the Navier--Stokes or the Boussinesq systems, physical parameters such as the Reynolds, Raleigh or Prandtl numbers determine whether or not the system approaches an equilibrium or transitions to chaos, while in the SEIRD model of transmission of an epidemic, the associated parameters determine whether the disease is eradicated or becomes endemic. Other examples include reservoir modelling \cite{Ig} and determining the transmissivity coefficient of groundwater, which captures its ability to move across an aquifer; it is a central parameter in these models that must be estimated in some way \cite{EL91, Yeh86} in order to make accurate predictions.
 
Due to its ubiquity in applications, there has been a recent surge of interest in parameter determination and estimation problems \emph{from finitely many partial observations} of the system under consideration. For instance, in case of a geophysical model with a state variable $\theta (\mbf x,t)$, depending on a spatial variable $\mbf x$ and time $t$, such observations may take the form of {\em nodal} values $\{\theta(\mbf x_i,t)\}_{i=1}^N$ where $\{\mbf x_i\}_{i=1}^N$ is a finite set of points in the spatial domain. These observations can either be collected continuously in time or at a discrete sequence of time points $t_1< t_2< \cdots$. 

Other examples of commonly used observational data include finitely many Fourier coefficients $\{\widehat{\theta}(n,t):|n| \le N\}$ or {\em volume elements} $\{\bar{\theta}_i\}_{i=1}^N$ where $\bar{\theta}_i=\int_{\Omega_i} \theta(\mbf x, t)\}$ and $\Omega_i$ is a subset of the whole spatial domain $\Omega$. For finite dimensional dynamical systems, such  partial observations constitute a time series of observations on a subset of the variables. In the SEIRD model of epidemic transmission, data may not be available for certain variables such as the exposed population due to asymptomatic transmission, thus leading to partial observations in these cases. 

The above mentioned examples can be mathematically formulated as follows. Let $\{\bu (t)\}$ be a dynamical system parameterized by $\nu$, evolving on a Hilbert space, $\mathcal H$, according to the equation
\be \label{eqn:param}
\dt{\bu}=F_{\mbf \nu}(\bu(t)), \: \bu(\cdot) \in \mathcal H,\: {\mbf \nu} \in \R^d.
\ee
The problem under consideration then is the estimation of the parameter ${\mbf \nu} \in \R^d$ based on {\em observations} $\{\co (\bu(t))\}$, where $\co: \cal H \lra \cal \R^N$ is an adequate linear (or nonlinear) operator and $t \in [t_0,T], T \le \infty$. More generally, one may  have observations on a discrete set of points $\{\co(\bu(t_i)\}_{i=1}^\infty$ and the observations may be contaminated with error.
 When ${\mbf \nu}$ is known, recovering the state variable from the observational data is commonly referred to as data assimilation. When ${\mbf \nu}$ is unknown, we may view this as a parameter estimation problem embedded in data assimilation.
 This has led to various computational strategies using Bayesian estimation methods \cite{dashti2017bayesian,Cotter_2009,xun2013parameter}, machine learning \cite{RK2018, RKP2019,ayed2019learning,baumeister1997line}, the Kalman filter \cite{evensen2009ensemble,van2001square}, or other numerical methods based on nudging \cite{carlson2018,PWM,albanez_parameter_2022,DiLeoni_Clark_Mazzino_Biferale_2018_inferring}, as well as software packages implementing solutions (e.g. \cite{raue2015data2dynamics,ruthotto2017jinv,hippylib}).
In this context, we note first that the parameter identification problem is not in general well-posed even when the state variable is fully observable, i.e. when the solution $\bu$ in \eqref{eqn:param} is known exactly. As described below, this can be seen in the multi-parameter identification problem addressed in \cite{PWM}.

Let $\Omega=[0,2\pi]$ and 
$$\cal H=\{f:\Omega \ra \R,\: f\ \text{is}\ 2\pi-\text{periodic},\: \int_{\Omega} f=0\}$$
and $A=-\partial_{xx}$ on $\cal H \cap \sH^2$. Consider now 
the problem of determining the parameter vector, $(\lambda_1,\lambda_2)$, from the solution, $u$, of the equation
\be  \label{paramillposed}
u_t - \lambda_1 A u + \lambda_2 A^2 u =f,
\ee
which is an illustrative example of the type of parameter identification problems discussed in \cite{PWM}.
If one takes $u=u_\delta$ to be an (time-independent) eigenfunction of $A$ corresponding to the eigenvalue $\delta$, and $f=c\delta u_\delta$, then $u_\delta$ solves the equation \eqref{paramillposed} provided $(\lambda_1-\lambda_2 \delta)=c$. However, \emph{one cannot determine the parameters} $\lambda_1$ and $\lambda_2$, even if one knows $u$ completely on the entire domain, $\Omega$; in a neighborhood of this solution, the inverse problem of determining the parameter $(\lambda_1,\lambda_2)$ is  ill-posed. Thus, the question of whether or not the parameter, ${\mbf \nu}$, can be determined in \eqref{eqn:param} from partial observations (i.e. whether or not the  parameter-to-data map is one-to-one) is even more delicate. 

\subsection{Overview of Results}
Here, for the specific case of determining the viscosity parameter for the two-dimensional Navier--Stokes equations (2D NSE), we consider the fundamental questions inherent in rigorous justification of the parameter estimation  methods mentioned before, namely, when can one recover the parameter from observational data and what is the regularity property of the inverse map. First, we find a lower bound on the resolution of the data (the number of Fourier modes) which guarantees that the viscosity is uniquely determined by the data over an interval of time, and then we show that this inverse mapping satisfies a Lipschitz property.

In addition, we define an algorithm for solving the inverse problem using the determining map. Our algorithm can be seen as an improvement over the algorithm originally presented in \cite{carlson2018} due to the replacement of the \emph{a posteriori} verifiable condition with an \emph{a priori} condition for the update to have a nonzero denominator. Our rigorous convergence criteria is also an improvement over the recent work in \cite{Martinez_2022} (which was developed simultaneously with this work) for the same reason; in addition, our criteria is independent of the magnitude of the observable error, whereas the condition in \cite{Martinez_2022} requires more data as the observable error decreases.

To give more context to our results, consider \cite{CialencoGlattHoltz2011}, where the authors considered the stochastically perturbed Navier--Stokes equations and constructed estimators which converge in probability to the true viscosity in the limit as the number of modes collected from a given sample path tends to infinity. In comparison, we show that, for the deterministically forced Navier--Stokes equations, the viscosity is determined from a {\em finite} number of modes.

In Section~\ref{sec:preliminaries} we will establish some notation and review the basic theory necessary for studying the Navier--Stokes equations. We will also briefly review the nudging data assimilation algorithm originally presented in \cite{AOT}. Then, in Section~\ref{sec:determining-map}, we define the determining map, extending its domain from that of \cite{BFMT} to include viscosity as a variable, and prove some of its regularity properties.

In Section~\ref{sec:parameter-recovery}, we frame the viscosity recovery inverse problem as a PDE constrained optimization problem, and provide rigorous conditions for the inverse problem to have a unique solution. Then, in Section~\ref{sec:nonzero-loss}, we prove that the loss function controls the viscosity error, which we leverage in Section~\ref{sec:algorithm} to define an algorithm that converges to the true viscosity, solving the inverse problem. 

\section{Preliminaries}\label{sec:preliminaries}
\subsection{The Navier--Stokes equations and their functional form} 
	
The Navier--Stokes equations for a homogeneous, incompressible Newtonian fluid in two dimensions are given by
	\be \label{nse}
	\frac{\partial \bu}{\partial t} - \nu \Delta \bu + (\bu \cdot \nabla ) \bu + \nabla p = \g, \quad \nabla \cdot \bu = 0,
	\ee
	where $\bu = (u_1,u_2)$ and $p$ are the unknowns and denote the velocity vector field and the pressure, respectively, while $\nu > 0$ and $\g$ are given and denote the kinematic viscosity parameter and the body forces applied to the fluid per unit mass, respectively (see \cite{Doering_Gibbon_1995_book} for an introduction to the Navier--Stokes equations). In this work, we consider the Navier--Stokes equations defined on the spatial domain
		\[ \Omega = [0,L]\times[0,L], \] 
	with periodic boundary conditions. 
	
	Let $\hat{\bu}(\mbf k), \mbf k=(k_1,k_2) \in \Z^2$ denote the $\mbf k$-th Fourier coefficient of $\bu$.  Given finitely many Fourier coefficients $\{\hat \bu (\mbf k):|\mbf k| \le N\}$, we address the question as to when the viscosity $\nu$ be determined from this finite observational data and whether the data-to-viscosity map is Lipschitz.

	We now proceed to recall the basic functional setting of the NSE, a systematic development of which can be found in \cite{cf, Temambook1997, Temambook2001}.
	Let $\mV$ be the space of test functions, given by
	\begin{multline}
	\mV = \left\{\bs{\varphi}: \R^2 \to \R^2 \, :\, \bs{\varphi}\, \text{ is an $L$-periodic trigonometric polynomial, } \right.\\
	\left. \nabla \cdot \bs{\varphi} = 0, \int_\Omega \bs{\varphi} (\bx) \rd \bx = 0 \right\}.
	\end{multline}
	We denote by $H$ and $V$ the closures of $\mV$ with respect to the norms in $(L^2(\Omega))^2$ and $(H^1(\Omega))^2$, respectively. Moreover, we denote by $H'$ and $V'$ the dual spaces of $H$ and $V$, respectively. As usual, we identify $H$ with $H'$, so that $V \subseteq H \subseteq V'$ with the injections being continuous and compact, with each space being densely embedded in the following one. The duality action between $V'$ and $V'$ is denoted by $\langle \cdot, \cdot \rangle_{V',V}$.

	The inner product in $H$ is given by
	\[
	(\bu_1,\bu_2) = \int_\Omega \bu_1 \cdot \bu_2 \, \rd \bx \quad \forall \bu_1, \bu_2 \in H,
	\]
	with the corresponding norm denoted by $|\bu| := \|\bu\|_{L^2} = (\bu,\bu)^{1/2}$. In $V$, we consider the following inner product:
	\[
	(\!( \bu_1,\bu_2 )\!) = \int_\Omega \nabla \bu_1 : \nabla \bu_2 \, \rd \bx \quad \forall \bu_1, \bu_2 \in V,
	\]
	where it is understood that $\nabla \bu_1 : \nabla \bu_2$ denotes the component-wise product between the tensors $\nabla \bu_1$ and $\nabla \bu_2$. The corresponding norm in $V$ is given by $\| \bu \| := \|\nabla \bu\|_{L^2} = (\!( \bu, \bu )\!)^{1/2}$. The fact that $\| \cdot \|$ defines a norm on $V$ follows from the Poincar\'e inequality, given in \eqref{eqn:Poincare}, below.

	For every subspace $\Lambda \subset (L^1(\Omega))^2$, we denote
	\[
	\dot{\Lambda}_\per = \left\{ \bs{\varphi} \in \Lambda \,:\, \bs{\varphi} \mbox{ is $L$-periodic and } \int_{\Omega} \bs{\varphi}(\bx) \rd \bx = 0 \right\}.
	\]

	Observe that $H$ is a closed subspace of $(\dot{L}^2(\Omega))^2$.
	Let $P_\sigma$ denote the Helmholtz-Leray projector, which is defined as the orthogonal projection from $(\dot{L}^2_\per(\Omega))^2$ onto $H$. Applying $P_\sigma$ to \eqref{nse}, we obtain the following equivalent functional formulation of the Navier--Stokes equations.
	\begin{system}[Navier--Stokes in functional form]
		\be\label{sys:Navier--Stokes}
		\frac{\rd \bu}{\rd t} + \nu A \bu + B(\bu,\bu) = \f \mbox{ in } V',
		\ee
	where $\f = P_\sigma \g$.
	\end{system}
	The bilinear operator $B: V \times V \to V'$ is defined as the continuous extension of
	\[
	B(\bu,\bv) = P_\sigma [(\bu \cdot \nabla) \bv] \quad \forall \bu, \bv \in \mV,
	\]
	and $A: D(A) \subset V \to V'$, the Stokes operator, is the continuous extension of
	\[
	A \bu = - P_\sigma \Delta \bu \quad \forall \bu \in \mV.
	\]
	In fact, in the case of periodic boundary conditions, we have $A = - \Delta$.

	We recall that $D(A) = V \cap (\dot{H}^2_\per(\Omega))^2$ and that $A$ is a positive and self-adjoint operator with compact inverse. Therefore, the space $H$ admits an orthonormal basis $\{\bs \bphi_j\}_{j=1}^\infty$ of eigenfunctions of $A$ corresponding to a non-decreasing sequence of eigenvalues $\{\lambda_j\}_{j=1}^\infty$,
	where $\lambda_j \in \{\kappa_0^2|\bk|^2, \bk \in \Z^2 \setminus \{0\}\}$ and $\lambda_1:=\kappa_0^2=(2\pi/L)^2$. 
	For a periodic function $\bphi$ on $\Omega$, we recall the Parseval's identity:
	\[
		 |\bphi|^2 = L^2 \sum_\bk |\hat{\bphi}_\bk|^2,
	\]
	where $\hat{\bphi}_\bk$ denotes the Fourier coefficient corresponding to $\bk=(k_1,k_2) \in \Z^2$.
	Furthermore, $\bphi \in L^2(\Omega)$ satisfies the mean-free condition $\int_\Omega \bphi =0$ if and only if $\hat{\bphi}_0=0$.
	  
	Given a normed space $\mathcal{X}$ with norm $\|\cdot\|_\mathcal{X}$, let \( C_b(\mathcal{X}) \) denote the continuous bounded mappings from $\mathbb{R}$ to $X$, and define the norm $\|\bphi\|_{C_b(X)} = \sup_{t\in\mathbb{R}} \|\bphi(t)\|_\mathcal{X}.$ In particular, we have the norms $\|\cdot\|_{C_b(H)}$ and $\|\cdot\|_{C_b(V)}.$

\subsection{Global attractor for the Navier--Stokes equations} 
	It is well-known that, given $\bu_0 \in H$, there exists a unique solution $\bu$ of \eqref{sys:Navier--Stokes} on $[0,\infty)$ such that $\bu(0) = \bu_0$ and
	\be\label{propssolNSE}
	\bu \in C([0,\infty);H) \cap \Lloc^2([0,\infty);V)\ \mbox{and}\ \dt \bu \in L^2_{loc}([0,\infty); V').
	\ee
	Moreover, we also have $\bu \in C((0,\infty);D(A))$ (see, e.g., \cite[Theorem 12.1]{cf}). Therefore, equation \eqref{sys:Navier--Stokes} has an associated semigroup $\{S_\nu(t)\}_{t\geq 0}$, where, for each $t \geq 0$, $S_\nu(t): H \to H$ is the mapping given by
	\be\label{defSt}
	S_\nu(t)\bu_0 = \bu_\nu(t),
	\ee
	with $\bu_\nu$ being the unique solution of \eqref{sys:Navier--Stokes} on $[0,\infty)$ satisfying $\bu_\nu(0) = \bu_0$ and \eqref{propssolNSE}. For simplicity, we will drop the subscript $\nu$ as it will be understood from context and simply write $S(t)$ and $\bu(t)$ instead.

	Recall that a bounded set $\mB \subset H$ is called \emph{absorbing} with respect to $\{S(t)\}_{t \geq 0}$ if, for any bounded subset $B \subset H$, there exists a time $T = T(B)$ such that $S(t)B \subset \mB$ for all $t \geq T$. The existence of a bounded absorbing set for \eqref{sys:Navier--Stokes} is a well-known result; therefore, a global attractor $\mA_\nu$ of \eqref{sys:Navier--Stokes} exists, and is uniquely defined by any of the equivalent conditions given below.

	\begin{definition}[Global Attractor]
	Let $\mB \subset H$ be a bounded absorbing set with respect to $\{S(t)\}_{t \geq 0}$.  
	Then the global attractor, $\mA_\nu$, exists as given by any of the following equivalent definitions (see \cite{RobinsonBook}):
	\begin{enumerate}
		\item \[ \mA_\nu = \bigcap_{t\geq0} S(t) \mB. \]
		\item $\mA_\nu$ is the largest compact subset of $H$ which is invariant under the action of the semigroup $\{S(t)\}_{t\geq 0}$, i.e., $S(t) \mA_\nu = \mA_\nu$ for all $t \geq 0$.
		\item $\mA_\nu$ is the minimal set that attracts all bounded sets.
		\item $\mA_\nu$ is the set of all points in $H$ through which there exists a globally bounded trajectory $\bu(t)$, $t \in \R$, with $\sup_{t \in \R} \|\bu(t)\|_{L^2} < \infty $.
	\end{enumerate}
	\end{definition}

	Also, recall the definition of the (dimensionless) Grashof number, given by
	\be\label{defG}
	G_\nu = \frac{\|\f\|_{L^2}}{(\nu \kappa_0)^2}.
	\ee
	In the periodic case, the following bounds hold on the global attractor, $\mA_\nu$:
	\be\label{boundmAH1}
	\|\bu\|_{L^2} \le \nu G_\nu, \quad \|\nabla \bu\|_{L^2} \leq \nu \kappa_0 G_\nu \quad \forall \bu \in \mA_\nu,
	\ee
	and
	\be\label{boundmADA}
uu	\|A\bu\|_{L^2} \leq c_2  \nu \kappa_0^2 (G_\nu + c_0^{-2})^3 \quad \forall \bu \in \mA_\nu,
	\ee
	where $c_0$ is the constant given below in \eqref{ineqLadyzhenskaya}, and $c_2 = 2137\,c_0^4$. The proof of \eqref{boundmAH1} can be found in any of the references listed above (\cite{cf, Temambook1997, Temambook2001}), and the proof of \eqref{boundmADA} is given in \cite[Lemma 4.4]{FoiasJollyLanRupamYangZhang2015}.

	In particular, we will make use of the following fact wich follows directly from \eqref{boundmAH1} and \eqref{boundmADA}: if $\bu\in\mA_\nu$, then $\bu \in C_b(H)$ and $\bu \in C_b(V)$. 

\subsection{Data Assimilation via Nudging} 
	The data assimilation problem we consider is defined as follows: given measurements of a reference solution of \eqref{sys:Navier--Stokes} on the attractor, and a guess for the fluid viscosity, contruct an approximation of the solution. This kind of problem is typically known as data assimilation, but can also be thought of as an inverse problem (mapping measurements of the velocity field back to the full velocity field) or a regression problem (interpolating the measurements in a way that generalizes to the entire space and time domain). 

	There are many approaches to solving data assimilation problems. Extensions of the Kalman filter for nonlinear problems, like the extended Kalman filter, or ensemble Kalman filter are common \cite{evensen2009ensemble,van2001square}, as well as 3DVAR and 4DVAR \cite{BlomkerLawStuartZygalakis2013}, or more recently, machine learning approaches such as Physics Informed Neural Networks \cite{RK2018,RKP2019}. The approach we consider, often called nudging, benefits from a rigorous mathematical framework for establishing convergence developed by Azouani, Olsen and Titi in 2014 \cite{AOT}. Since then, many authors have developed extensions of nudging and applied nudging to several equations (e.g. \cite{albanez2016continuous,Altaf_Titi_Knio_Zhao_Mc_Cabe_Hoteit_2015,bessaih2015continuous,Biswas_Bradshaw_Jolly_2020,biswas2018continuous,Biswas_Martinez_2017,Biswas_Price_2020_AOT3D,blocher2018data,Desamsetti_Dasari_Langodan_Knio_Hoteit_Titi_2019_WRF,Celik_Olson_Titi_2019,Chen_Li_Lunasin_2021,farhat2020data,Farhat_Jolly_Titi_2015,farhat2016charney,Farhat_Lunasin_Titi_2016abridged,Farhat_Lunasin_Titi_2016benard,Farhat_Lunasin_Titi_2017_Horizontal,Foias_Mondaini_Titi_2016,Gardner_Larios_Rebholz_Vargun_Zerfas_2020_VVDA,ibdah2020fully,Jolly_Martinez_Olson_Titi_2018_blurred_SQG,Jolly_Martinez_Titi_2017,Larios_Pei_2018_NSV_DA,Larios_Victor_2019,Larios_Victor_2021_chiVsdelta2D,Markowich_Titi_Trabelsi_2016_Darcy,Pei_2019,Larios_Pei_2017_KSE_DA_NL}), and studied the efficacy of nudging numerically \cite{Gesho_Olson_Titi_2015,Lunasin_Titi_2015,di2020synchronization,Hudson_Jolly}.

	The idea can be simply stated as follows: given data collected on a variable, $y$, and a method of interpolating the data, putting them together, we have an interpolant operator $I_h$, where $h$ is a measure of the resolution of the data, and $I_h(y)$ is the interpolation of the data. If $y$ satisfies a differential equation $\dot{y} = F(y)$, then the nudging algorithm is to solve $$ \dot{\tilde{y}} = F(y) + \mu \left(I_h(y) - I_h(\tilde{y})\right) $$
	for an approximation, $\tilde{y}$. The algorithm is succesful when $\tilde{y}$ converges to $y$.
	
	In \cite{AOT}, the authors prove results for interpolant operators satisfying one of two approximation criteria. The first kind of interpolant operator (and most strict) has the following property:
	\begin{definition}[Type~1 Interpolant Operator]
	\[ |\bphi - I_h(\bphi)| \leq c_1 h |\nabla \bphi| \quad \forall \bphi \in (\dot{H}^1(\Omega))^2 \]
	\end{definition}
	In this work, we focus on the case where the measurements are exact values for finitely many spectral coefficients, which are then interpolated to give approximations in the range of the modal projection, \(\Proj.\) Specifically, for any $N \geq 1$, \(\Proj\) is defined by:
	\[ \widehat{\Proj(\bphi)}_\bk = \begin{cases} \widehat{\bphi}_\bk, \quad |\bk| < N \\ 0, \quad |\bk| \geq N \end{cases} \quad \forall \bk \in \mathbb{Z}^2. \]
	
	The modal projection is an example of a Type~1 interpolant operator, as can easily be verified.
	\begin{equation}\label{eqn:Type-1}
		|\bphi - \Proj{\bphi}|^2 \leq \frac{1}{\kappa_0^2}\frac{1}{N^2} \|\bphi\|^2
	\end{equation}
	\optionalproof{ 
	\begin{multline*}	
		|\bphi|^2 = \int_{\Omega} |\bphi|^2 d\bx 
		= \int_{\Omega} \sum_\bk\sum_\bl \bphi_\bk \cdot \bar{\bphi_\bl} e^{i \kappa_0 \bk\cdot \bx - i \kappa_0 \bl\cdot \bx} d\bx \\
		= \sum_\bk\sum_\bl (\bphi_\bk \cdot \bar{\bphi_\bl}) \int_{\Omega} e^{i \kappa_0 (\bk-\bl)\cdot \bx} d\bx
		= \sum_\bk\sum_\bl (\bphi_\bk \cdot \bar{\bphi_\bl}) |\Omega| \delta_{\bk \bl} \\
		= |\Omega| \sum_\bk |\bphi_\bk|^2 = |\Omega| |\hat{\bphi}|^2 
	\end{multline*}
	and
	\begin{multline*}
		|\bphi - \Proj{\bphi}|^2 = |\Omega| |\hat{\bphi} - \widehat{\Proj{\bphi}}|^2 
			\leq |\Omega| \frac{1}{N^2} \sum_{\bk\in\mathbb{Z}^2} |\bk|^2|\hat{\bphi}_\bk|^2 \\
			= |\Omega| \frac{1}{\kappa_0^2 N^2} \sum_{\bk\in\mathbb{Z}^2} |i \kappa_0 \bk \hat{\bphi}_\bk^T|^2 
			= |\Omega| \frac{1}{\kappa_0^2 N^2} |\widehat{\nabla \bphi}|^2 
			= \frac{1}{\kappa_0^2}\frac{1}{N^2} |\nabla \bphi|^2
	\end{multline*}
	}
	In particular, for any \(\bphi \in H\) (which by the mean-free condition implies \(P_{1}{\bphi} = 0\)) we have the Poincar\'e inequality
	\begin{equation}\label{eqn:Poincare}
		|\bphi| = |\bphi - P_{1}{\bphi}| \leq \frac{1}{\kappa_0} \| \bphi \|.
	\end{equation}

	We explicitly state the nudging data assimilation algorithm for our current setting in System~\ref{sys:data-assimilation}.
	\begin{system}[Navier--Stokes with data feedback]
		Given $\bphi \in C_b(V)$ (the data) and an approximate viscosity, $\gamma$, find $\bv \in C_b(V)$ satisfying
		\begin{equation}\label{sys:data-assimilation}
			\NSE{\bv}{\gamma} = \f + \mu (\bphi - \Proj{\bv}).
		\end{equation}
	\end{system}
	When solving the data assimilation problem, we typically have $\bphi = \Proj{\bu}$, where $\bu\in\mA_\nu$. However, \eqref{sys:data-assimilation} remains valid in the general setting with $\bphi \in C_b(V)$, and is the basis for the definition of the determining map in Section~\ref{sec:determining-map}.

\subsection{Standard Inequalities}
	We now review several standard inequalities which are used throughout the subsequent sections.
	We often use Young's inequality to establish estimates,
	\begin{equation}\label{eqn:Young-full}
		a^p b^{1-p} \leq p a + (1-p) b,\quad \forall a,b \geq 0, \; p\in[0,1].
	\end{equation}
	In particular, we frequently use Young's inequality in the following form:
	\begin{equation}\label{eqn:Young}
		ab \leq \frac{1}{2\epsilon}a^2 + \frac{\epsilon}{2}b^2,\quad \forall a,b \geq 0, \; \epsilon > 0.
	\end{equation}
	We also make use of Jensen's inequality,
	\begin{equation}\label{eqn:Jensen}
		\left(\frac{1}{L}\int_0^L \bphi(x) dx\right)^2 \leq \frac{1}{L}\int_0^L \bphi^2(x) dx
	\end{equation}
	
	We use Ladyzhenskaya's inequality (for periodic boundary conditions) in the following form to obtain estimates of the nonlinear term in \eqref{sys:data-assimilation} (a proof is provided in the Appendix):
	\begin{equation}\label{eqn:Lady}
		\|\bphi\|_{L^4(\Omega)} \leq |\bphi|^{\frac12} \left(\tfrac{1}{L} |\bphi| + \|\bphi\|\right)^{\frac12}
	\end{equation}
	\begin{remark}
		Note that, using \eqref{eqn:Poincare}, we can rewrite \eqref{eqn:Lady} in the more usual form
		\be  \label{ineqLadyzhenskaya}
		 \|\bphi\|_{L^4(\Omega)} \leq c_0 |\bphi|^{\frac12} \|\bphi\|^{\frac12}
		\ee
		where $c_0 = \sqrt{1 + \tfrac1{2\pi}}$. 
	\end{remark}

	Using \eqref{eqn:Lady}, we have the following estimate:
	\begin{multline}\label{eqn:nonlinear-est-Lady}
		|\b{\bu}{\bv}{\bw}| \leq \|\bu\|_{L^4(\Omega)} |\nabla \bv| \|\bw\|_{L^4(\Omega)} \\
		\leq |\bu|^{\frac12} \left(\tfrac{1}{L} |\bu| + \|\bu\|\right)^{\frac12}
		\|\bv\|
		|\bw|^{\frac12} \left(\tfrac{1}{L} |\bw| + \|\bw\|\right)^{\frac12}
	\end{multline}
	where $\b{\bu}{\bv}{\bw} := \left(\B{\bu}{\bv},{\bw}\right)$.


\section{Determining Map}\label{sec:determining-map}
The determining map was first defined in \cite{FJKT1,FJKT2} as the mapping of data to the corresponding solution of \eqref{sys:data-assimilation}. It was exploited in \cite{BFMT} for statistical data assimilation. We extend the definition of the determining map to include viscosity as an input and establish its Lipschitz property in all its arguments, which plays a pivotal role in establishing the results in Section~\ref{sec:parameter-recovery} and Section~\ref{sec:algorithm}.

	\begin{definition}[Determining Map]\label{def:detmap} 
		Fix an arbitrary lower bound $\nu_0 > 0$ for the viscosity, and fix an arbitrary radius $R > 0$ for a ball $B_R(0) \subset C_b(V)$. The determining map, $\mbb W:(\nu_0,\infty)\times B_R(0) \to C_b(V)$, is the mapping of viscosity and data to the corresponding solution of \eqref{sys:data-assimilation} on the attractor.
	\end{definition} 

	\noindent For example, $\Wmap{\tilde{\nu}}{\Proj{\bu}}$ is the solution of \eqref{sys:data-assimilation} with $\gamma = \tilde{\nu}$ and $\bphi = \Proj{\bu}$. However, we don't in general require that the data $\bphi$ come from a solution of \eqref{sys:Navier--Stokes}.

	We begin by deriving sufficient conditions for the determining map to be a well-defined Lipschitz continuous mapping. These conditions are placed on $\mu$ and $N$. Note that in contrast to the usual analysis taken wherein $N$ depends on $\mu$, we derive our results in such a way that the lower bound for $N$ comes directly from the data, forcing, and $\nu_0$, and $\mu$ is chosen to be larger than a lower bound which grows with $N^2$. Specifically, we parameterize $\mu$ in terms of a (non-dimensional) constant \( \mu_0 > 0 \) as follows:
	\begin{equation}\label{eqn:mu-param}
		\mu = \mu_0 \nu_0 \kappa_0^2 N^2.
	\end{equation}

	\begin{lemma}\label{lem:apriori-bounds} 
		Fix \( \nu_0 \in \mathbb{R}_+ \) and suppose that 
		\begin{equation}\label{cnd:mu-N}
			\mu_0 \geq 1 \quad\text{i.e.}\quad \mu \geq \nu_0 \kappa_0^2 N^2.
		\end{equation}
		For any \(\f,\bphi \in C_b(V), \) and for any \( 1 \leq N \in \mathbb{R}, \)  if 
		\[ \gamma \in [\nu_0, \infty) \]
		and $\bv$ is a corresponding solution of \eqref{sys:data-assimilation}, then the following bounds are satisfied:
		\begin{equation}\label{prp:W-map_bounded-H}
			\|\bv\|_{C_b(H)}^2 \leq \frac{2}{(\nu_0 \kappa_0^2 N^2)^2}\|\f\|_{C_b(H)}^2 
				+ 2\mu_0^2 \|\bphi\|_{C_b(H)}^2 \le (M_H(\bphi))^2,
		\end{equation}
		and
		\begin{equation}\label{prp:W-map_bounded-V}
			\|\bv\|_{C_b(V)}^2 \leq \frac{2}{(\nu_0 \kappa_0^2 N^2)^2}\|\f\|_{C_b(V)}^2 
				+ 2\mu_0^2 \|\bphi\|_{C_b(V)}^2\le (M_V(\bphi))^2.
		\end{equation}
		Here, given fixed \(\f \in C_b(V),\) \(\nu_0, \kappa_0 > 0,\) and $\mu_0 \geq 1$, we define
		\begin{equation}\label{eqn:M_H}
		M_H(\bphi) = \sqrt{2}\sqrt{\frac{1}{(\nu_0 \kappa_0^2 )^2}\|\f\|_{C_b(H)}^2 + \mu_0^2 \|\bphi\|_{C_b(H)}^2}
		\end{equation}
		and 
		\begin{equation}\label{eqn:M_V}
		M_V(\bphi) = \sqrt{2}\sqrt{\frac{1}{(\nu_0 \kappa_0^2 )^2}\|\f\|_{C_b(V)}^2 + \mu_0^2 \|\bphi\|_{C_b(V)}^2}.
		\end{equation}
	\end{lemma} 
	\begin{proof}
		Let \(\bv\) be a solution of \eqref{sys:data-assimilation}. Then taking the inner product of \eqref{sys:data-assimilation} with $\bv$, we have
		\[
			\frac12 \frac{d}{dt} |\bv|^2 + \gamma \|\bv\|^2
			= \left<\f, \bv\right> + \mu \left<\bphi, \bv\right> - \mu |\Proj{\bv}|^2.
		\]
		\def\epsorig{c_1}
		\def\deltorig{c_2}
		Therefore, for any choice of \( \epsorig, \deltorig > 0 \),
		\begin{multline*}
			\frac12 \frac{d}{dt} |\bv|^2 + \gamma \|\bv\|^2 + \mu |\Proj{\bv}|^2
			\leq |\f| |\bv| + \mu |\bphi| |\bv|
			\leq \frac{1}{2\epsorig} |\f|^2 + \frac{\mu}{2\deltorig} |\bphi|^2
				+ \frac{\epsorig + \mu\deltorig}{2}  |\bv|^2 \\
			\leq \frac{1}{2\epsorig} |\f|^2 + \frac{\mu}{2\deltorig} |\bphi|^2
				+ \frac{\epsorig + \mu\deltorig}{2}  |\Proj{\bv}|^2 
				+ \frac{1}{\kappa_0^2}\frac{\epsorig + \mu\deltorig}{2N^2}  \|\bv\|^2,
		\end{multline*}
		where to obtain the last inequality, we used $|\bv|^2=|P_N\bv|^2+|(I-P_N)\bv|^2$ and
		\eqref{eqn:Type-1}. 
		Collecting terms, we have
		\[
			\frac12 \frac{d}{dt} |\bv|^2
				+ \left(\gamma - \frac{1}{\kappa_0^2}\frac{\epsorig + \mu\deltorig}{2N^2}\right) \|\bv\|^2
				+ \left(\mu - \frac{\epsorig + \mu\deltorig}{2}\right) |\Proj{\bv}|^2
			\leq \frac{1}{2\epsorig} |\f|^2 + \frac{\mu}{2\deltorig} |\bphi|^2.
		\]

		Let \( \epsilon, \delta \in (0,1), \) 
		and define 
		\[ r = \frac{\nu_0 \kappa_0^2 N^2}{\mu}. \]
		Choosing
		\[
			\epsorig = 2\mu r\delta(1 - \epsilon)  \quad \text{ and } \quad \deltorig = 2r(1-\delta),
		\]
		we have 
		\[ 
			\epsorig + \mu \deltorig  = 2\mu r(1 - \delta\epsilon)
		\]
		and the previous differential inequality becomes
		\be  \label{eq:basicdiffineq}
			\frac12 \frac{d}{dt} |\bv|^2
				+ \left(\gamma - \nu_0(1-\delta\epsilon)\right) \|\bv\|^2
				+ \left(\mu - \nu_0 \kappa_0^2 N^2(1 - \delta\epsilon)\right) |\Proj{\bv}|^2
			\leq \frac{1}{2\epsorig} |\f|^2 + \frac{\mu}{2\deltorig} |\bphi|^2.
		\ee
		Observe that
		\[
		\|\bv\|^2=\|P_N\bv\|^2+\|(I-P_N)\bv\|^2
		\ge \kappa_0^2 N^2 |(I-P_N)\bv|^2,
		\]
		where to obtain the last inequality we used \eqref{eqn:Type-1}.
Consequently, 
		\begin{multline}
			\frac12 \frac{d}{dt} |\bv|^2
				+ \nu_0\delta\epsilon \|\bv\|^2
				+ \left(\mu - \nu_0 \kappa_0^2 N^2(1 - \delta\epsilon)\right) |\Proj{\bv}|^2
			\\ \geq 
			\frac12 \frac{d}{dt} |\bv|^2
				+ \nu_0\kappa_0^2 N^2 \delta\epsilon |\bv|^2
				+ \left(\mu - \nu_0 \kappa_0^2 N^2\right) |\Proj{\bv}|^2.  \label{eq:diffineq2}
		\end{multline}
		Note that by imposing the constraint \(\epsilon, \delta \in (0,1),\) we have ensured that 
		\( \gamma - \nu_0(1 - \delta\epsilon) \geq \nu_0\delta\epsilon > 0.  \)
		Using \eqref{eq:basicdiffineq}, \eqref{eq:diffineq2} and  \eqref{cnd:mu-N}, and by dropping the last term in  \eqref{eq:diffineq2}, we obtain
		\[
			\frac{d}{dt} |\bv|^2 + \beta |\bv|^2
			\leq \frac{1}{\epsorig} |\f|^2 + \frac{\mu}{\deltorig} |\bphi|^2,
		\]
		with \( \beta := 2\nu_0\kappa_0^2 N^2 \delta\epsilon. \)

		So, by Gr\"owall's inequality,
		\begin{multline*}
			|\bv(t)|^2 \leq e^{-\beta (t-s)} |\bv(s)|^2 
			+ \int_s^t e^{-\beta(t-\tau)} \left( 
				\frac{1}{\epsorig} |\f(\tau)|^2 + \frac{\mu}{\deltorig} |\bphi(\tau)|^2
			\right) d\tau \\
			\leq e^{-\beta (t-s)} |\bv(s)|^2 
			+ \left( \frac{1}{\epsorig} \|\f\|_{C_b(H)}^2 + \frac{\mu}{\deltorig} \|\bphi\|_{C_b(H)}^2 \right) 
				\frac{1}{\beta}\left(1 - e^{-\beta(t-s)}\right)
		\end{multline*}
		We get the stated result by choosing 
		\( \epsilon = \frac{2}{3} \) 
		and 
		\( \delta = \frac{3}{4} \) 
		(so \( \epsorig = \frac12 \nu_0 \kappa_0^2 N^2 \) and \( \deltorig = \frac12 \frac{\nu_0 \kappa_0^2 N^2}{\mu} \)),
		and taking the limit as \( s\to -\infty. \)
	
		Taking the inner product of \eqref{sys:data-assimilation} with $\A \bv$ and proceeding similarly, but this time using the enstrophy cancelation property $(B(\bv,\bv),A\bv)=0$, we obtain \eqref{prp:W-map_bounded-V}.
	\end{proof}

	\comments{ 
	For convenience in using the upper bounds given in Lemma~\ref{lem:apriori-bounds}, given fixed \(\f \in C_b(V),\) \(\nu_0, \kappa_0 > 0,\) and $\mu_0 \geq 1$, we define
	\begin{equation}\label{eqn:M_H}
		M_H(N,\bphi) = \sqrt{2}\sqrt{\frac{1}{(\nu_0 \kappa_0^2 N^2)^2}\|\f\|_{C_b(H)}^2 + \mu_0^2 \|\bphi\|_{C_b(H)}^2}
	\end{equation}
	and 
	\begin{equation}\label{eqn:M_V}
		M_V(N,\bphi) = \sqrt{2}\sqrt{\frac{1}{(\nu_0 \kappa_0^2 N^2)^2}\|\f\|_{C_b(V)}^2 + \mu_0^2 \|\bphi\|_{C_b(V)}^2},
	\end{equation}
	for any $N \geq 1$ and $\bphi \in C_b(V)$.
	Note that $M_H$ and $M_V$ are finite and nonincreasing in $N$, 
	so inequalities of the form $N > M_H(N, \bphi)$ and $N > M_V(N, \bphi)$ are always feasible and can be satisfied by increasing $N$.
	}

	We now derive the main Lipschitz continuity property of the determining map which is essential to the proofs of the subsequent results.

	\begin{lemma}\label{lem:Lipschitz} 
		Let \(\bv_1\) and \(\bv_2\) be solutions of \eqref{sys:data-assimilation} with viscosity \( \gamma_1 \geq \nu_0 \) and \( \gamma_2 \geq \nu_0 \) and data \( \bphi_1 \) and  \( \bphi_2 \) respectively.  Let \( p \in [0,1] \) and set
		\[ \alpha = \frac{\gamma_1 + \gamma_2}{2}\left(1 - \frac12 \left(2 \frac{|\gamma_1 - \gamma_2|}{\gamma_1 + \gamma_2}\right)^p \right). \]
		Then, $\alpha \in  \left[\frac12 \nu_0,\bar{\gamma}\right]$ where $\bar{\gamma}:= \frac12(\gamma_1 + \gamma_2)$. Moreover, 
		if
		\begin{equation}\label{cnd:mu-alpha}
			\mu_0 \geq \max\left\{1, \frac{\alpha}{\nu_0}\right\},
		\end{equation}
		and
		\begin{equation}\label{cnd:Lipschitz-N}
		N \geq \frac{4}{\alpha\kappa_0}
			M_V( \bphi_i), i=1,2.
		\end{equation}
		then
		\begin{equation}\label{eqn:Lipschitz}
			\|\bv_1 - \bv_2\|_{C_b(H)}^2 \leq \frac{5}{2} 
				\frac{|\gamma_1 - \gamma_2|^{2-p}}{\left(\frac{\gamma_1 + \gamma_2}{2}\right)^{1-p}}
				\frac{\|\frac{1}{2}(\bv_1 + \bv_2)\|_{C_b(V)}^2}{\alpha\kappa_0^2N^2} 
			+ \frac{5}{2} \mu_0^2 \left(\frac{\nu_0}{\alpha}\right)^2
			\|\bphi_1 - \bphi_2\|_{C_b(H)}^2.
		\end{equation}
	\end{lemma} 
	\begin{proof} 
		Observe that
		\[ \alpha = \bar{\gamma} - \frac12 \bar{\gamma}^{1-p} |\gamma_1 - \gamma_2|^{p}
		= \bar{\gamma}\left(1 - \frac12 \left(2 \frac{|\gamma_1 - \gamma_2|}{\gamma_1 + \gamma_2}\right)^p \right), \] 
		and note that for all \( p\in[0,1], \)
		\begin{equation}\label{eqn:alpha-bounds}
		\bar{\gamma} \geq \alpha \geq \bar{\gamma}\left(
		1 - \frac12 \max\left\{1, 2 \frac{|\gamma_1 - \gamma_2|}{\gamma_1 + \gamma_2}\right\}
		\right) 
		= \min\left\{\frac12\bar{\gamma}, \min\left\{\gamma_1, \gamma_2\right\}\right\}
		\geq \frac12 \nu_0.
		\end{equation}
		Thus,  $\alpha \in  \left[\frac12 \nu_0,\bar{\gamma}\right]$

		Now let \(\bv_1\) and \(\bv_2\) be solutions of \eqref{sys:data-assimilation} with viscosity \( \gamma_1 \) and \( \gamma_2 \) and data \( \bphi_1 \) and  \( \bphi_2 \) respectively. We now analyze the difference \( \bw:= \bv_1 - \bv_2, \) and for convenience of presentation, we define the average quantities \( \bar{\bv} = \frac12(\bv_1 + \bv_2) \) and \( \bar{\gamma} = \frac12(\gamma_1 + \gamma_2). \) 

		Writing the evolution equation for \( \bw \) using \eqref{sys:data-assimilation},
		\[ 
			\partial_t \bw + \bar{\gamma} \A \bw + (\gamma_1 - \gamma_2) \A \bar{\bv} 
			+ \B{\bar{\bv}}{\bw} + \B{\bw}{\bar{\bv}} = \mu (\bphi_1 - \bphi_2) - \mu\Proj{\bw}
		\]
		and taking the inner product with \( \bw \) and using \eqref{eqn:nonlinear-est-Lady}, we have
		\begin{multline*}
			\frac12 \frac{d}{dt}|\bw|^2 + \bar{\gamma} \|\bw\|^2 + \mu|\Proj{\bw}|^2 
			= - (\gamma_1 - \gamma_2) \left<\A \bar{\bv},\bw\right> + \mu \left<(\bphi_1 - \bphi_2), \bw\right> - \b{\bw}{\bar{\bv}}{\bw} \\
			\leq |\gamma_1 - \gamma_2| \|\bar{\bv}\| \|\bw\| + \mu |\bphi_1 - \bphi_2| |\bw| + \|\bar{\bv}\||\bw|(\tfrac1L|\bw| + \|\bw\|)
		\end{multline*}

		Using \eqref{eqn:Young}, 
		for any $p\in[0,1]$, we can write
		\[
			|\gamma_1 - \gamma_2| \|\bar{\bv}\| \|\bw\| 
			\leq \frac12 \bar{\gamma}^{p-1} |\gamma_1 - \gamma_2|^{2-p} \|\bar{\bv}\|^2 
				+ \frac12 \bar{\gamma}^{1-p} |\gamma_1 - \gamma_2|^{p} \|\bw\|^2,
		\]
		therefore,
		\begin{multline*}
			\frac12 \frac{d}{dt}|\bw|^2 
			+ \left(
				\bar{\gamma} - \frac12 \bar{\gamma}^{1-p} |\gamma_1 - \gamma_2|^{p}
			\right) \|\bw\|^2 
			+ \mu |\Proj{\bw}|^2 \\
			\leq \frac12 \bar{\gamma}^{p-1} |\gamma_1 - \gamma_2|^{2-p} \|\bar{\bv}\|^2 
				+ \mu |\bphi_1 - \bphi_2| |\bw| + \|\bar{\bv}\||\bw|(\tfrac1L|\bw| + \|\bw\|).
		\end{multline*}
	

		For the remaining terms, using \eqref{eqn:Young} and \eqref{eqn:Type-1}, for any \( c_1 > 0 \) we have
		\[ 
			\mu |\bphi_1 - \bphi_2| |\bw| 
			\leq \frac{\mu}{2 c_1} |\bphi_1 - \bphi_2|^2 
				+ \frac{\mu c_1}{2 \kappa_0^2 N^2} \|\bw\|^2 + \frac{\mu c_1}{2} |\Proj{\bw}|^2, 
		\]
		and 
		\verbose{ 
		\begin{multline*}
			\|\bar{\bv}\||\bw|(\tfrac1L|\bw| + \|\bw\|) 
			= \tfrac1L\|\bar{\bv}\||\bw|^2 + \|\bar{\bv}\||\bw|\|\bw\| \\
			\leq \frac{1}{\kappa_0^2 N^2 L}\|\bar{\bv}\|\|\bw\|^2 + \frac1L\|\bar{\bv}\||\Proj{\bw}|^2 
				+ \frac{c_2}{2}\|\bar{\bv}\||\bw|^2 + \frac{1}{2c_2}\|\bar{\bv}\|\|\bw\|^2 \\
			\leq \frac{1}{\kappa_0^2 N^2 L}\|\bar{\bv}\|\|\bw\|^2 + \frac1L\|\bar{\bv}\||\Proj{\bw}|^2 
				\\ + \frac{c_2}{2}\|\bar{\bv}\||\Proj{\bw}|^2 + \frac{c_2}{2\kappa_0^2 N^2}\|\bar{\bv}\|\|\bw\|^2 + \frac{1}{2c_2}\|\bar{\bv}\|\|\bw\|^2 \\
			\leq \frac{1}{\kappa_0^2 N^2 L}\|\bar{\bv}\|\|\bw\|^2 + \frac1L\|\bar{\bv}\||\Proj{\bw}|^2 
				\\ + \frac{\kappa_0 N}{2}\|\bar{\bv}\||\Proj{\bw}|^2 + \frac{1}{2\kappa_0 N}\|\bar{\bv}\|\|\bw\|^2 + \frac{1}{2\kappa_0 N}\|\bar{\bv}\|\|\bw\|^2 \\
			\leq \left(\frac{1}{\kappa_0^2 N^2 L} + \frac{1}{\kappa_0 N}\right)\|\bar{\bv}\|\|\bw\|^2 
				+ \left(\frac1L + \frac{\kappa_0 N}{2}\right) \|\bar{\bv}\||\Proj{\bw}|^2.
		\end{multline*}
		} 
		\begin{multline*}
			\|\bar{\bv}\||\bw|(\tfrac1L|\bw| + \|\bw\|) \\
				\leq \left(\frac{1}{\kappa_0^2 N^2 L} + \frac{1}{\kappa_0 N}\right)\|\bar{\bv}\|\|\bw\|^2 
				+ \left(\frac1L + \frac{\kappa_0 N}{2}\right) \|\bar{\bv}\||\Proj{\bw}|^2.
		\end{multline*}
		Therefore,
		\verbose{ 
		\begin{multline*}
			\frac12 \frac{d}{dt}|\bw|^2 + \alpha \|\bw\|^2 + \mu|\Proj{\bw}|^2
			\leq \frac12 \bar{\gamma}^{p-1} |\gamma_1 - \gamma_2|^{2-p} \|\bar{\bv}\|^2
				+ \frac{\mu}{2 c_1} |\bphi_1 - \bphi_2|^2  \\
			+ \frac{\mu c_1}{2 \kappa_0^2 N^2} \|\bw\|^2
			+ \left(\frac{1}{\kappa_0^2 N^2 L} + \frac{1}{\kappa_0 N}\right)\|\bar{\bv}\|\|\bw\|^2 \\
			+ \frac{\mu c_1}{2} |\Proj{\bw}|^2
			+ \left(\frac1L + \frac{\kappa_0 N}{2}\right) \|\bar{\bv}\||\Proj{\bw}|^2
		\end{multline*}
		} 
		\begin{multline*}
			\frac12 \frac{d}{dt}|\bw|^2
			+ \left(
				\alpha
				- \frac{\mu c_1}{2 \kappa_0^2 N^2}
				- \left(\frac{1}{\kappa_0^2 N^2 L} + \frac{1}{\kappa_0 N}\right)\|\bar{\bv}\|
			\right) \|\bw\|^2 \\
			+ \left(
				\mu
				- \frac{\mu c_1}{2}
				- \left(\frac1L + \frac{\kappa_0 N}{2}\right) \|\bar{\bv}\|
			\right) |\Proj{\bw}|^2 \\
			\leq \frac12 \bar{\gamma}^{p-1} |\gamma_1 - \gamma_2|^{2-p} \|\bar{\bv}\|^2
				+ \frac{\mu}{2 c_1} |\bphi_1 - \bphi_2|^2
		\end{multline*}
		and after choosing \( c_1 = \alpha \frac{\kappa_0^2 N^2}{\mu} \) and simplifying,
		\begin{multline*}
			\frac12 \frac{d}{dt}|\bw|^2
			+ \left( \frac12 \alpha 
				- \left(1 + \frac{1}{2\pi N}\right)\frac{\|\bar{\bv}\|}{\kappa_0 N}
			\right) \|\bw\|^2 \\
			+ \left( \mu
				- \frac{1}{2}\alpha \kappa_0^2 N^2
				- \frac1L\left(1 + \pi N\right) \|\bar{\bv}\|
			\right) |\Proj{\bw}|^2 \\
			\leq \frac12 \bar{\gamma}^{p-1} |\gamma_1 - \gamma_2|^{2-p} \|\bar{\bv}\|^2
				+ \frac{\mu^2}{2 \alpha \kappa_0^2 N^2} |\bphi_1 - \bphi_2|^2.
		\end{multline*}

		Now, by \eqref{cnd:Lipschitz-N} and \eqref{prp:W-map_bounded-V}, we have
		\begin{equation}\label{eqn:N-controls-v}
			N \geq \frac{4}{\alpha\kappa_0} \|\bar{\bv}\|,
		\end{equation}
		so, 
		\begin{equation}\label{eqn:Lipschitz-rem-diss} 
		 \frac12 \alpha - \left(1 + \frac{1}{2\pi N}\right)\frac{\|\bar{\bv}\|}{\kappa_0 N} 
			\geq \frac12 \alpha - \left(1 + \frac{1}{2\pi N}\right)\frac{\alpha}{4}
			= \frac14 \alpha \left(1 - \frac1{2\pi N}\right)
			> \frac15 \alpha 
			> 0. 
		\end{equation}
		Therefore, we can apply \eqref{eqn:Type-1}, obtaining
		\begin{multline*}
			\frac12 \frac{d}{dt}|\bw|^2
			+ \left( \frac12 \alpha 
				- \left(1 + \frac{1}{2\pi N}\right)\frac{\|\bar{\bv}\|}{\kappa_0 N}
			\right) \kappa_0^2N^2 |\bw|^2 \\
			+ \left( \mu
				- \alpha \kappa_0^2N^2 
				+ \frac12 \kappa_0 N \|\bar{\bv}\|
			\right) |\Proj{\bw}|^2 \\
			\leq \frac12 \bar{\gamma}^{p-1} |\gamma_1 - \gamma_2|^{2-p} \|\bar{\bv}\|^2
				+ \frac{\mu^2}{2 \alpha \kappa_0^2 N^2} |\bphi_1 - \bphi_2|^2.
		\end{multline*}
		By assumption \eqref{cnd:mu-alpha} 
		we can drop the third term. Then, after simplifying and applying the estimate \eqref{eqn:Lipschitz-rem-diss} to the second term on the left hand side, we have
		\[
			\frac{d}{dt}|\bw|^2
				+ \frac{2}{5} \alpha \kappa_0^2 N^2 |\bw|^2 \\
			\leq \bar{\gamma}^{p-1} |\gamma_1 - \gamma_2|^{2-p} \|\bar{\bv}\|^2 
				+ \frac{\mu^2}{\alpha \kappa_0^2 N^2} |\bphi_1 - \bphi_2|^2.
		\]
		From here, for any \( -\infty < s < t \) we can apply Gr\"onwall's inequality, 
		\begin{multline*}
			|\bw(t)|^2 \leq e^{-\frac{2}{5} \alpha \kappa_0^2 N^2 (t-s)} |\bw(s)|^2  \\
			+ \bar{\gamma}^{p-1} |\gamma_1 - \gamma_2|^{2-p} \|\bar{\bv}\|_{C_b(V)}^2
				\frac{5}{2}\frac{1}{\alpha\kappa_0^2N^2} \left(1 - e^{-\frac{2}{5} \alpha \kappa_0^2 N^2(t-s)}\right)\\
			+ \frac{\mu^2}{\alpha \kappa_0^2 N^2} \|\bphi_1 - \bphi_2\|_{C_b(H)}^2 
				\frac{5}{2}\frac{1}{\alpha\kappa_0^2N^2} \left(1 - e^{-\frac{2}{5} \alpha \kappa_0^2 N^2(t-s)}\right)
		\end{multline*}
		and taking the limit as \( s\to-\infty \) we get the stated result.
	\end{proof} 

	In the previous lemma, the lower bound for $N$ depends on the value of $\alpha$, which depends on the viscosities. If we impose an upper bound, $\nu_1$, on the viscosities in addition to the lower bound, $\nu_0$, we can easily replace this dependence to be on $\nu_1$ and $\nu_0$.

	\begin{corollary}\label{cor:Lipschitz_wUpperBound} 
		Let \(\bv_1\) and \(\bv_2\) be solutions of \eqref{sys:data-assimilation} with viscosity \( \gamma_1 \) and \( \gamma_2 \) and data \( \bphi_1 \) and  \( \bphi_2 \) respectively, and suppose that \( \gamma_1, \gamma_2 \in [\nu_0, \nu_1] \), where \( \nu_1 > \nu_0 \).
		If
		\begin{equation}\label{cnd:mu-cor0}
			\mu_0 \geq \frac{\nu_1}{\nu_0},
		\end{equation}
		and
		\begin{equation}\label{cnd:N-cor0}
		N \geq \frac{8}{\nu_0\kappa_0}
			M_V( \bphi_i)
		\end{equation}
		for \(i = 1,2,\)
		then
		\begin{equation*}
			\|\bv_1 - \bv_2\|_{C_b(H)}^2 \leq 5
				\frac{|\gamma_1 - \gamma_2|^{2-p}}{\left(\frac{\gamma_1 + \gamma_2}{2}\right)^{1-p}}
				\frac{\|\frac{1}{2}(\bv_1 + \bv_2)\|_{C_b(V)}^2}{\nu_0\kappa_0^2N^2} 
			+ 10 \mu_0^2 \|\bphi_1 - \bphi_2\|_{C_b(H)}^2.
		\end{equation*}
	\end{corollary}
	\begin{proof}
		Substituting for $\alpha$ using the bounds 
		\( \frac{\nu_0}{2} \leq \alpha \leq \frac{\gamma_1 + \gamma_2}{2} \leq \nu_1 \)
		proves the corollary.
	\end{proof} 

	We can also allow for $\gamma_1$ and $\gamma_2$ to be arbitrarily large, as long as their relative difference is smaller than $1$.
	
	\begin{corollary}\label{cor:Lipschitz_close} 
		Let \(\bv_1\) and \(\bv_2\) be solutions of \eqref{sys:data-assimilation} with viscosity \( \gamma_1 \geq \nu_0 \) and \( \gamma_2 \geq \nu_0 \) and data \( \bphi_1 \) and  \( \bphi_2 \) respectively.
		If
		\begin{equation}\label{cnd:mu-cor1}
			\mu_0 \geq 1,
		\end{equation}
		and
		\begin{equation}\label{cnd:N-cor1}
		N \geq \frac{8}{\nu_0\kappa_0}
			M_V(\bphi_i)
		\end{equation}
		for \(i = 1,2,\)
		and if 
		\[ \frac{|\gamma_1 - \gamma_2|}{\frac12(\gamma_1 + \gamma_2)} \leq 1, \]
		then
		\begin{equation*}
			\|\bv_1 - \bv_2\|_{C_b(H)}^2 \leq 5 
				\left(\frac{|\gamma_1 - \gamma_2|}{\frac{\gamma_1 + \gamma_2}{2}}\right)^{2-p}
				\frac{\|\frac{1}{2}(\bv_1 + \bv_2)\|_{C_b(V)}^2}{\kappa_0^2N^2} 
			+ 10 \mu_0^2 \|\bphi_1 - \bphi_2\|_{C_b(H)}^2.
		\end{equation*}
	\end{corollary}
	\begin{proof}
	Using the same notation as before, we have by assumption \( \frac{|\gamma_1 - \gamma_2|}{\bar{\gamma}} \leq 1, \) so
	\( \frac12 \nu_0 \leq \alpha \leq \nu_0, \) hence \eqref{cnd:mu-alpha} and \eqref{cnd:Lipschitz-N} are satisfied by \eqref{cnd:mu-cor1} and \eqref{cnd:N-cor1}. Also,
		\begin{equation*}
		\alpha \geq \bar{\gamma}\left(
			1 - \frac12 \max\left\{1, 2 \frac{|\gamma_1 - \gamma_2|}{\gamma_1 + \gamma_2}\right\}
		\right) 
			= \frac12 \bar{\gamma}.
		\end{equation*}
	Substituting this lower bound for $\alpha$ in \eqref{eqn:Lipschitz} along with the more general lower bound, $\alpha > \frac12\nu_0$, gives the result.
	\end{proof} 
	
	We can now show the well-posedness of the determining map.

	\begin{theorem}\label{thm:determining-map} 
		Let $R > 0$ and $\nu_0 > 0$, and let \( \f \in C_b(V). \) 
		If
		\begin{equation}\label{cnd:N-det-thm}
		N \geq \frac{8}{\nu_0\kappa_0}
			\sqrt{2}\left(\frac{1}{\nu_0 \kappa_0^2 N^2}\|\f\|_{C_b(V)} 
				+ \mu_0 R\right),
		\end{equation}
		for any $\mu_0 \geq 1$ with $\mu$ given by \eqref{eqn:mu-param}, then 
		$\W$
		as defined by Definition~\ref{def:detmap} is well-defined.
	\end{theorem}
	\begin{proof}
		The existence of solutions on the attractor for any given pair $(\gamma, \bphi)$ on the domain of $\W$ can be shown by first taking a Galerkin truncation and proceeding to the limit as is done in the classical case for the 2D Navier--Stokes, and so we omit the details here.

		The uniqueness of solutions follows from Corollary~\ref{cor:Lipschitz_close}: for any $(\gamma, \bphi)\in(\nu_0,\infty)\times B_R(0)$, note that the conditions of the corollary are satisfied. Therefore, given two solutions $\bv_1$ and $\bv_2$ of \eqref{sys:data-assimilation} corresponding to $\gamma$ and $\bphi$, by Corollary~\ref{cor:Lipschitz_close} we have
		\[ \|\bv_1 - \bv_2\|_{C_b(H)} = 0, \]
		therefore, by the embedding $V\subset H$,
		\[ \|\bv_1 - \bv_2\|_{C_b(V)} = 0. \]
	\end{proof} 

	\begin{remark}
	Theorem~\ref{thm:determining-map} also provides an alternate proof of the classical result that there are finitely many determining modes for the Navier--Stokes equations. To see that this is true, suppose \( \bu \) and \( \bv \) are two solutions of \eqref{sys:Navier--Stokes} in $C_b(V)$ with viscosity $\nu$, which agree at low modes, i.e. \( \Proj{\bu} = \Proj{\bv} \) for all $t\in\mathbb{R}$. If $N$ satisfies the conditions of Theorem~\ref{thm:determining-map} with 
	\( \nu_0 = \nu, \) 
	\( R = \max\{\|\bu\|_{C_b(V)}, \|\bv\|_{C_b(V)}\}, \)
	and $\mu_0 = 1$, then
	\[ \bu = \Wmap{\nu}{\Proj{\bu}} = \Wmap{\nu}{\Proj{\bv}} = \bv. \]
	\end{remark}


\section{The Parameter Recovery Inverse Problem}\label{sec:parameter-recovery}
	Our goal in this section is to define and study the following optimization problem.
	\begin{problem}\label{pbl:main-optimization}
		Let $\Proj{\bu}$ be given on $\mathbb{R}$. Define the cost functional, $\Cost$, by
		\begin{equation}
			\Cost(\gamma) = \| \Proj{\Wmap{\gamma}{\Proj{\bu}}} - \Proj{\bu} \|_{C_b(H)}.
		\end{equation}
		Find $\gamma^* > 0$ such that
		\begin{equation*}
			\Cost(\gamma^*) = \min_{\gamma > 0} \Cost(\gamma).
		\end{equation*}
	\end{problem}

	As motivation, consider the following:
	\begin{fact}\label{sec:Motivation}
	If $\bu \in C_b(V)$ is a solution of \eqref{sys:Navier--Stokes} with viscosity $\nu$, then
	\begin{equation*}
		\min_{\gamma > 0} \| \Wmap{\gamma}{\Proj{\bu}} - \bu \|_{C_b(H)}
	\end{equation*}
	has a unique solution, namely $\nu$, provided that $\mu$ and $N$ satisfy the conditions of Theorem~\ref{thm:determining-map} with \( R = \|\bu\|_{C_b(V)}, \) and that $\f \not\equiv 0$. 
	\end{fact}
	\begin{proof}
		Choosing $\gamma = \nu$ makes the cost exactly $0$, as $\bu$ solves \eqref{sys:data-assimilation} when $\gamma = \nu$, and therefore, by Theorem~\ref{thm:determining-map}, $\Wmap{\nu}{\Proj{\bu}} = \bu$. To see that this solution is unique, suppose that 
		\( \bv:= \Wmap{\gamma}{\Proj{\bu}} \) satisfies \( \| \bv - \bu \|_{C_b(H)} = 0. \)
		Then  
		\[ \NSE{\bv}{\gamma} = \f + \Proj{\bu} - \Proj{\bv}. \]
		Taking an inner-product with $\bv$, we obtain
		\[ \frac12 \frac{d}{dt} |\bv|^2 + \gamma \|\bv\|^2 = \left<\f, \bv\right> + \left<\Proj{\bu}, \bv\right> - |\Proj{\bv}|^2, \]
		whereas, by the definition of $\bu$ we have
		\[ \NSE{\bu}{\nu} = \f, \]
		and taking an inner-product with $\bu$, we get
		\[ \frac12 \frac{d}{dt} |\bu|^2 + \nu \|\bu\|^2 = \left<\f, \bu\right>. \]

		Now, by assumption, $|\bv - \bu| = 0$ for each $t$, so
		$\left<\f, \bv\right> = \left<\f, \bu\right>$,
		$\left<\Proj{\bu}, \bv\right> = |\Proj{\bv}|^2$,
		and
		\( \|\bu\| = \|\bv\|. \)
		Therefore, taking the difference of the two energy equations, we see that for all $t \in (-\infty, \infty)$,
		\[ (\gamma - \nu) \|\bu\| = 0. \] 
		This can only hold if $\bu \equiv 0$ or if $\gamma = \nu$. However, $\bu \equiv 0$ only if 
		$\f \equiv 0$. Therefore $\gamma = \nu$.
	\end{proof}
	
	Note that the minimization problem presented in Fact~\ref{sec:Motivation} is not useful, because to calculate the cost we would need to know $\bu$ exactly. In which case, rather than solving the minimization problem using $\Wmap{\cdot}{\cdot}$, we could compute $\nu$ directly by averaging the energy equation over $[-T, T]$ for any $T > 0$:
	\begin{equation}\label{eqn:nu-direct-estimate}
		\nu = \frac{ \int_{-T}^{T} \left<\f, \bu\right> - \left(|\bu(T)|^2 - |\bu(-T)|^2\right)}{\int_{-T}^{T} \|\bu\|^2}.
	\end{equation}
	The point of posing Problem~\ref{pbl:main-optimization} is to obtain a similar result while requiring a minimal amount of data. 
	Note that \eqref{eqn:nu-direct-estimate} could be modified by replacing $\bu$ with $\Proj \bu$, which is essentially the approach taken in \cite{CialencoGlattHoltz2011}, but for the stochastically forced NSE. In contrast, rather than use the data $\Proj \bu$ directly, we leverage the convergence properties of the determining map with finite $N$ to solve for $\nu$.

\section{Uniqueness of the Inverse Problem}
	Given the same data, the Lipschitz properties we proved in Section~\ref{sec:determining-map} show that closeness in the viscosity implies smallness of the cost function, $\Cost$. Now that we are considering the inverse problem, we are interested in the converse: when does smallness in the cost function imply closeness to the true viscosity? We give conditions for the converse to be true in Section~\ref{sec:nonzero-loss}, but first, we directly address the simpler question of uniqueness of the solution, $\gamma = \nu$, to Problem~\ref{pbl:main-optimization}.

	In the following lemma, we provide an inequality bounding the difference between two viscosities which, given the same data $\bphi$, map (via $\W$) to trajectories with identical projections on \( range(\Proj{}). \) We leverage this result to show the uniqueness of solutions of Problem~\ref{pbl:main-optimization} in the subsequent theorem.
	\begin{lemma}\label{lem:viscosity-error} 
		Let $\f,\bphi\in C_b(V)$ and \( \gamma_1, \gamma_2 \in [\nu_0, \nu_1]\), and fix
		\( \mu_0 \geq \frac{\nu_1}{\nu_0} \) and 
		\( N \geq \frac{8}{\nu_0\kappa_0} M_V(\bphi). \)
		If
		\begin{equation}
			\|\Proj{\Wmap{\gamma_1}{\bphi}} - \Proj{\Wmap{\gamma_2}{\bphi}}\|_{C_b(H)} = 0,
		\end{equation}
		then either \( |\gamma_1 - \gamma_2| = 0, \) or for any $p\in[0,1]$, 
		\begin{equation}
			\left(\frac{|\gamma_1 - \gamma_2|}{\frac{\gamma_1 + \gamma_2}{2}}\right)^{p/2}
			\|\bpsi\|_{C_b(H)}\, N
			\leq  \frac{2 \sqrt{5c}}{\nu_0\kappa_0^2 \sqrt{|\Omega|} } 
				\sqrt{\ln(N+1)} 
				M_H( \bphi) M_V( \bphi),
		\end{equation}
		where \( \bpsi := \Proj{\Wmap{\gamma_1}{\bphi}} = \Proj{\Wmap{\gamma_2}{\bphi}}. \)
	\end{lemma}
	\begin{proof}
		Let $\bv_1 = \Wmap{\gamma_1}{\bphi}$ and $\bv_2 = \Wmap{\gamma_2}{\bphi}$, and assume that $\Proj{\bv_1} = \Proj{\bv_2} =: \bpsi$. Let $\bw = \bv_1 - \bv_2$ (note that $\Proj{\bw}$ is therefore $0$). 
		Writing \eqref{sys:data-assimilation} in Fourier space, for any $\bk\in \mathbb{Z}^2$, $|\bk| \leq N$, we have
		\[
			\NSEFourier{\bv_1}{\gamma_1} = \F{\f}(\bk) + \mu \F{\bphi}(\bk) - \mu \F{\bv_1}(\bk)
		\]
		and
		\[
			\NSEFourier{\bv_2}{\gamma_2} = \F{\f}(\bk) + \mu \F{\bphi}(\bk) - \mu \F{\bv_2}(\bk).
		\]
		Subtracting these equations and using the fact that $\F{\bv_1}(\bk) = \F{\bv_2}(\bk) = \F{\bpsi}(\bk)$ and $\F{\bw}(\bk) = 0$ for $|\bk| \leq N$, we find
		\[
			(\gamma_1 - \gamma_2)\kappa_0^2|\bk|^2\F{\bpsi}(\bk) + \conv{\bv_1}{\bw} + \conv{\bw}{\bv_2} = 0.
		\]

		Next we estimate the convolutions. First, we may write
		\[
			|\gamma_1 - \gamma_2| \kappa_0 |\bk|^2 |\F{\bpsi}(\bk)| \leq 
			\sum_{\bh} |\bk| |\F{\bv_1}(\bh)| |\F{\bw}(\bk-\bh)| + \sum_{\bh} |\bk| |\F{\bv_2}(\bk-\bh)| |\F{\bw}(\bh)|.
		\]
		Then, applying the Cauchy-Schwarz inequality and Parseval's theorem, we obtain
		\[
			|\gamma_1 - \gamma_2| \kappa_0 |\bk|^2 |\F{\bpsi}(\bk)| 
				\leq |\bk|\frac{1}{|\Omega|} |\bv_1| |\bw| + |\bk| \frac{1}{|\Omega|}|\bv_2| |\bw|.
		\]
		Now, for any $\bk \neq 0$, we can multiply both sides by $\frac{|\F{\bpsi}(\bk)|}{|\bk|^2}|\Omega|$,
		\[
			|\gamma_1 - \gamma_2| \kappa_0 |\F{\bpsi}(\bk)|^2 |\Omega|
				\leq \frac{1}{|\bk|}|\F{\bpsi}(\bk)| (|\bv_1| + |\bv_2|) |\bw|,
		\]
		and taking the sum over $\bk\in\mathbb{Z}^2, 0 < |\bk| \leq N,$ and using the Cauchy-Schwarz inequality,
		\[
			|\gamma_1 - \gamma_2| \kappa_0 |\Proj{\bpsi}|^2
				\leq \left(\sum_{0 < |\bk| \leq N}\frac{1}{|\bk|^2}\right)^\frac{1}{2} \frac{1}{\sqrt{|\Omega|}}|\Proj{\bpsi}| (|\bv_1| + |\bv_2|) |\bw|.
		\]
		Using the estimate
		\[
			\sum_{0 < |\bk| \leq N}\frac{1}{|\bk|^2} \leq 8 + 2\pi(1 + \ln(N)) \leq c \ln(N + 1),
		\]
		and noting that \( \bpsi = \Proj{\bpsi} \), we have
		\[
			|\gamma_1 - \gamma_2| \kappa_0 \sqrt{|\Omega|} |\bpsi|^2
				\leq \sqrt{c \ln(N+1)} |\bpsi| (|\bv_1| + |\bv_2|) |\bw|.
		\]

		Finally, by Corollary~\ref{cor:Lipschitz_wUpperBound},
		\[
		|\bw(t)| \leq \sqrt{5}
				\frac{|\gamma_1 - \gamma_2|^{1-p/2}}{\left(\frac{\gamma_1 + \gamma_2}{2}\right)^{1/2-p/2}}
				\frac{\|\frac{1}{2}(\bv_1 + \bv_2)\|_{C_b(V)}}{\sqrt{\nu_0}\kappa_0 N},
		\]
		and using Lemma~\ref{lem:apriori-bounds}, we can bound the norms involving $\bv_1$ and $\bv_2$. The resulting inequality is
		\begin{multline*}
			\left(\frac{|\gamma_1 - \gamma_2|}{\frac{\gamma_1 + \gamma_2}{2}}\right)^{p/2}
			|\gamma_1 - \gamma_2| |\bpsi|^2
			\\
			\leq  |\gamma_1 - \gamma_2| |\bpsi| 
				\frac{2 \sqrt{5}}{\nu_0\kappa_0^2 \sqrt{|\Omega|} } 
				\frac{\sqrt{c \ln(N+1)}}{N} 
				M_H( \bphi) M_V( \bphi),
		\end{multline*}
		so either \( |\gamma_1 - \gamma_2| = 0 \) (taking care in the case where $p = 0$), or
		\[
			\left(\frac{|\gamma_1 - \gamma_2|}{\frac{\gamma_1 + \gamma_2}{2}}\right)^{p/2}
			|\bpsi| N
			\leq  \frac{2 \sqrt{5c}}{\nu_0\kappa_0^2 \sqrt{|\Omega|} } 
				\sqrt{\ln(N+1)} 
				M_H( \bphi) M_V( \bphi).
		\]
		This holds for all time, so we may take the supremum over all time and get the stated result.
	\end{proof} 

	In the previous lemma no conclusions can be drawn regarding the viscosity in the event that the data are zero in $C_b(H)$ (compare to the nonzero requirement of Fact~\ref{sec:Motivation}). However, given a nonzero function $\bphi \in C_b(H)$, we always have $\|P_n(\bphi)\|_{C_b(H)} > 0$ for some $n > 1$. Therefore, given $0 \neq \bphi \in C_b(H)$, let
	\begin{equation}  \label{eqn:smallestN}
		n_0(\bphi) = \sup \{n \in \mathbb{R} \:|\: \|P_n(\bphi)\|_{C_b(H)} = 0 \}.
	\end{equation}
	Then \( 1 \leq n_0(\bphi) < \infty \),  and \( \|P_{N}(\bphi)\|_{C_b(H)} > 0 \) for any \( N > n_0(\bphi). \) 

	We now give conditions for Problem~\ref{pbl:main-optimization} to have a unique solution. 
	\begin{theorem}\label{thm:optimization_well-posedness} 
		Let \( \bu \in C_b(V) \) be a solution of \eqref{sys:Navier--Stokes} with viscosity 
		\( \nu \in [\nu_0, \nu_1] \),
		where \( 0 < \nu_0 < \nu_1 < \infty \) are given, and suppose that $\bu \neq 0$. 
		Choose \( \mu_0 \geq \frac{\nu_1}{\nu_0}. \) 
		If
		\begin{equation*}
			N > \max\left\{\frac{8}{\nu_0\kappa_0} M_V(\bu),\: n_0(\bu)\right\}
		\end{equation*}
		and
		\begin{equation}\label{cnd:N-uniqueness}
			\frac{N}{\sqrt{\ln(N+1)}}
			> \frac{2 \sqrt{5c}}{\nu_0\kappa_0^2 \sqrt{|\Omega|} } 
				\frac{M_H( \bu) M_V( \bu)} {\|\Proj{\bu}\|_{C_b(H)}},
		\end{equation}
		then $\nu$ is the \emph{unique} solution of Problem~\ref{pbl:main-optimization} on \([\nu_0, \nu_1]\).
	\end{theorem} 
	\begin{proof} 
		By the assumption that \( \bu\in C_b(V), \)
		we may take 
		\( R = \|\Proj{\bu}\|_{C_b(V)} \leq \|\bu\|_{C_b(V)} < \infty, \)
		in Theorem~\ref{thm:determining-map} and conclude that
		\( \bu = \Wmap{\nu}{\Proj{\bu}}. \)

		Let 
		\( \gamma \in [\nu_0, \nu_1] \) 
		and suppose that
		\( \mathcal{L}(\gamma) = 0. \)
		Then by Lemma~\ref{lem:viscosity-error} with $p=0$, either $\nu = \gamma$ or
		\begin{multline*}
			\|\Proj{\bu}\|_{C_b(H)}\, N
			\leq  \frac{2 \sqrt{5c}}{\nu_0\kappa_0^2 \sqrt{|\Omega|} } 
				\sqrt{\ln(N+1)} 
				M_H( \Proj{\bu}) M_V( \Proj{\bu})
			\\
			\leq  \frac{2 \sqrt{5c}}{\nu_0\kappa_0^2 \sqrt{|\Omega|} } 
				\sqrt{\ln(N+1)} 
				M_H( \bu) M_V( \bu).
		\end{multline*}
		This inequality is invalidated by \eqref{cnd:N-uniqueness}, so we conclude that $\nu = \gamma$.
	\end{proof} 

	In Theorem \ref{thm:optimization_well-posedness}, condition \eqref{cnd:N-uniqueness} depends on the solution $\bu$ on which we only have knowledge of its finite-dimensional projection $P_Nu$. 
	The previous Theorem can be reformulated in terms of a dynamical property namely the distance between zero and the attractor, $\mathcal A$.

	\begin{theorem} \label{thm:attrzero} 
		Assume $\nu \in [\nu_0,\nu_1]$ with $\nu_0 >0$ and assume $\mbf 0 \notin \mA_\nu$.
		Let
		\[
		 G= \max\left\{ \frac{\|\f\|_{L^2}}{\nu_0^2\kappa_0^2},  \frac{\|\f\|_{\sH^1}}{\nu_0^2\kappa_0^3}\right\}\ \mbox{and}\  \delta=d_H(\mbf 0, \mathcal A_\nu)=\inf_{\bu \in \mA_\nu}|\bu|.
		 \]
		Then $n_0(\bu) \lesssim \frac{\nu_1G}{ \delta}$  where $G:= \frac{\|\f\|_{L^2}}{\nu_0^2\kappa_0^2}$.  Moreover, $\nu$ is the unique solution of Problem~\ref{pbl:main-optimization}
		provided
		\be \label{condn:attrzero-1}
		N \gtrsim \max\{\frac{\nu_1 G}{\delta}, \frac{\nu_1}{\nu_0} G\sqrt{1+  \frac{\nu_1}{\nu_0}}\}
		\ \mbox{and}\ \frac{N}{\sqrt{\ln(N+1)}} \gtrsim 
		\frac{\nu_1^2(1+ \frac{\nu_1}{\nu_0})G}{\nu_0\delta}.
		\ee
		\comments{	
			\be  \label{condn:attrzero-1}
			N \gtrsim \max\left\{\frac{\nu_1G}{ \delta}, \frac{M_V(\bu)}{\nu_0\kappa_0}\right\}
			\ \mbox{and}\ 
			\frac{N}{\sqrt{\ln(N+1) }} \gtrsim	\frac{M_H(\bu)M_V(\bu)}{\nu_0\kappa_0 \delta}.
			\ee
		}
	\end{theorem}
	\begin{proof}
		Note that in view of \eqref{eqn:Type-1}. we have
		\be  \label{keyineq}
		|P_n\bu -\bu|^2 \le \frac{1}{\kappa_0^2N^2}\|\bu\|^2\lesssim \frac{\nu_1^2G^2}{N^2},
		\ee
		where, with $G_\nu$ as in \eqref{defG}, we used the classical bounds \eqref{boundmAH1} yielding
		\be     \label{classicalbounds}
		\sup_{\bu \in \mathcal A_\nu} |\bu| \le  \nu G_\nu \le \nu_1G\ \mbox{and}\ 
		\sup_{\bu \in \mathcal A_\nu} \|\bu\| \le \nu \kappa_0 G_\nu \le  \nu_1 \kappa_0 G.
		\ee
		Then, provided $N \gtrsim \frac{\nu_1G}{ \delta}$, from \eqref{keyineq}, we have $|P_N\bu-\bu|^2 \le \frac12 \delta^2$. Using the equality $|\bu|^2=|P_N\bu|^2 + |(I-P_N)\bu|^2$ and the inequality $|\bu|^2 \ge \delta^2$ for $\bu \in \mathcal A_\nu$, it follows that $|P_N\bu|^2 \ge \frac12 \delta^2$. It follows from the definition \eqref{eqn:smallestN} that $n_0(\bu) \le  \frac{\nu_1G}{ \delta} $.  

			Let  $\mu_0 =\frac{\nu_1}{\nu_0}$. Since $\nu \in [\nu_0,\nu_1]$.
		 it follows readily from \eqref{classicalbounds} and the definition of $M_H(\bu)$ and $M_V(\bu)$ in \eqref{eqn:M_H} and \eqref{eqn:M_V} that
		\[
		M_H(\bu) \lesssim \nu_1 \sqrt{1+\frac{\nu_0}{\nu_1}}\, \widetilde{G}
		\ \mbox{and}\  M_H(\bu) \lesssim \nu_1 \sqrt{1+\frac{\nu_0}{\nu_1}}\kappa_0\, \widetilde{G}.
		\]
		The conclusion now  follows immediately from Theorem \ref{thm:optimization_well-posedness} by inserting the bound
		 $|P_N\bu|^2 \ge \frac12 \delta^2$ in \eqref{cnd:N-uniqueness}.
	\end{proof}

	\comments{
	\begin{remark}
		Let  $\mu_0 =\frac{\nu_1}{\nu_0}$ and and define the non-dimensional Grashoff number
		\[
		  \widetilde{G} = \max\left\{ \frac{\|\f\|_{L^2}}{\nu_0^2\kappa_0^2},  \frac{\|\f\|_{\sH^1}}{\nu_0^2\kappa_0^3}\right\}.
		\]
		Assume that $\nu \in [\nu_0,\nu_1]$.
		Then, it follows readily from \eqref{classicalbounds} and the definition of $M_H(\bu)$ and $M_V(\bu)$ in \eqref{eqn:M_H} and \eqref{eqn:M_V}, it follows that
		\[
		 M_H(\bu) \lesssim \nu_1 \sqrt{1+\frac{\nu_0}{\nu_1}}\, \widetilde{G}
		 \ \mbox{and}\  M_H(\bu) \lesssim \nu_1 \sqrt{1+\frac{\nu_0}{\nu_1}}\kappa_0\, \widetilde{G}.
		\]
		One can then express \eqref{condn:attrzero-1} as
		\[
		N \gtrsim \max\{\frac{\nu_1 \widetilde{G}}{\delta}, \frac{\nu_1}{\nu_0} \widetilde{G}\sqrt{1+  \frac{\nu_1}{\nu_0}}\}
		\ \mbox{and}\ \frac{N}{\sqrt{\ln(N+1)}} \gtrsim 
		\frac{\nu_1^2(1+ \frac{\nu_1}{\nu_0})\widetilde{G}}{\nu_0\delta}.
		\] 
	\end{remark}
	}

	We can easily see that the condition $\bu\neq 0$ is necessary for the unique recovery of $\nu$ by inspecting \eqref{sys:Navier--Stokes}: if $\bu \equiv 0$ satisfies \eqref{sys:Navier--Stokes} for a given force $\f$, then it does so with any value for $\nu$ because $Au \equiv 0$, hence there is no unique $\nu$ to recover. As evidenced by the following example, the condition $N > n_0(\bu)$ in Theorem~\ref{thm:optimization_well-posedness} is also necessary.

	\begin{example} 
	Fix \( \bk\in\mathbb{Z}^2/\{0\} \) and $c\in\mathbb{C}/\{0\}$ and define \( \f:\Omega\to V \) by 
	\[
		\f(\bx) = c \bk^\perp e^{i \kappa_0 \bk\cdot \bx} + \bar{c} \bk^\perp e^{-i \kappa_0 \bk\cdot \bx} 
		, \quad \forall \bx\in\Omega,
	\]
	(where $\bk^\perp$ is a unit vector orthogonal to $\bk$, for example, $[ -k_2 , k_1 ]^T / |\bk|$), i.e. $\f$ is a time-autonomous ``single mode'' force in $V$.
	Then
	\[ \bu \equiv \frac{1}{\nu\kappa_0^2|\bk|^2} \f \]
	is a stationary solution of \eqref{sys:Navier--Stokes} for any $\nu > 0$.
	Note that 
	\[
		\|\f\| 
		= 2\sqrt{2}\pi |\bk| |c|,
	\]
	so, by choosing $c = \frac{1}{2\sqrt{2}\pi |\bk|}$, we have $\| \f \| = 1$ for any choice of $\bk$, and
	\[
		\|\bu\| = \frac{1}{\nu\kappa_0^2|\bk|^2},
	\]
	so
	\[ 
		M_V(\Proj{\bu}) = \sqrt{2} \sqrt{ \frac{1}{(\nu \kappa_0^2 N^2)^2} + \mu_0^2\|\Proj{\bu}\|^2 } 
		\leq \frac{\sqrt{2}}{\nu \kappa_0^2}\sqrt{\frac{1}{N^4} + \mu_0^2 \frac{1}{|\bk|^4}}.
	\]

	Now, choosing $\mu_0 > 1$ and $N \geq \frac{8\sqrt{2}}{\nu^2 \kappa_0^3}\sqrt{1 + \mu_0^2}$, 
	for any $\gamma \in [\frac{\nu}{\mu_0}, \mu_0 \nu]$,
	\[ \Wmap{\gamma}{\Proj{\bu}} = \frac{1}{\gamma\kappa_0^2 |\bk|^2 + \mu}(\f + \mu \Proj{\bu}). \]
	Therefore,
	\[ \| \Wmap{\gamma}{\Proj{\bu}} - \bu \|_{C_b(H)} = \frac{|\f|}{\gamma\kappa_0^2 |\bk|^2 + \mu} \frac{1}{\nu}
		\begin{cases}
		|\gamma - \nu|, & \text{ if } |\bk| < N \\
		\left| \gamma - \nu + \frac{\mu}{\kappa_0^2 |\bk|^2} \right|, & \text{ if } |\bk| \geq N	
		\end{cases}
	\]
	and
	\[ 
		\| \Proj{\Wmap{\gamma}{\Proj{\bu}}} - \Proj{\bu} \|_{C_b(H)} = 
		\begin{cases}
			\frac{|\f|}{\gamma\kappa_0^2 |\bk|^2 + \mu} \frac{|\gamma - \nu|}{\nu}, & \text{ if } |\bk| < N  \\
			0, & \text{ if } |\bk| \geq N 
		\end{cases}.
	\] 
	So, if $|\bk| \geq N$, then \( \Proj{\bu} = 0, \) and \( \| \Proj{\Wmap{\gamma}{\Proj{\bu}}} - \Proj{\bu} \|_{C_b(H)} = 0 \) 
	while \( |\nu - \gamma| \) can be arbitrarily large. Hence, we have found cases with $N$ arbitrarily large and $\bu\neq0$ where there are infinitely many solutions of Problem~\ref{pbl:main-optimization} because $\Proj{\bu} = 0$.
	\end{example} 


\section{Lipschitz Continuity of the Inverse Problem}\label{sec:nonzero-loss}
	In the previous section we derived conditions for Problem~\ref{pbl:main-optimization} to have a unique solution. In so doing, we found conditions on $N$ for Problem~\ref{pbl:main-optimization} to be a well-defined mapping of data, $\Proj{\bu}$, to viscosity. We now provide conditions for this mapping, which can be thought of as the inverse of the determining map, $\Wmap{\cdot}{\Proj{\bu}}^{-1}$, to be Lipschitz continuous; i.e., we derive bounds on the viscosity discrepancy in terms of $\Cost$. In addition, we propose an algorithm for solving Problem~\ref{pbl:main-optimization}, which is revealed by the proof of the Lipschitz continuity.

	When $\Cost \geq 0$, the rate of change of the difference between two trajectories corresponding to different viscosities is dynamic, and needs to be integrated. As a consequence, the following results explicitly involve the time-derivative of the data, $\Proj \bu$.

	Before focusing on the case where $\phi = \Proj \bu$ for some solution $\bu$ of \eqref{sys:Navier--Stokes}, we first obtain a bound valid for the determining map in the general setting.
	\begin{lemma}\label{lem:viscosity-err-bound}
		Let \( \bphi \in C_b(V) \) and \( \gamma_1, \gamma_2 \in [\nu_0, \nu_1] \),
		where \( 0 < \nu_0 < \nu_1 < \infty \) are given.
		Let \( \mu_0 \geq \frac{\nu_1}{\nu_0}, \) and
		\begin{equation*}
			N > \frac{8}{\nu_0\kappa_0} M_V( \bphi).
		\end{equation*}
		Let $\bv_1 = \Wmap{\gamma_1}{\bphi}$ and $\bv_2 = \Wmap{\gamma_2}{\bphi}$.
		Then for any interval of time $[s, t] \subset \mathbb{R}$,
		\begin{equation*}
			|\gamma_1 - \gamma_2| \left( \inf_{[s, t]} |\Proj{\bv_1}|^2 - \frac{1}{N} M_1 \right)
			\leq M_2 \|\Proj (\bv_1 - \bv_2)\|_{C_b(H)},
		\end{equation*}
		where \[ 
			M_1 = \frac{ 2 \sqrt{5} c_0^2 }{ \nu_0 \kappa_0^2 }\; M_H(\bphi)\,(M_V(\bphi))^2
		\]
		and \[ 
			M_2 = \mu\frac{1 + \delta^{N^2}}{1-\delta^{N^2}}\|A^{-1} \Proj{\bv_1}\|_{C_b(H)} + \|A^{-1}\partial_t \Proj{\bv_1}\|_{C_b(H)} +\gamma_2 \|\Proj{\bv_1}\|_{C_b(H)},
		\]
		and $\delta = e^{-\mu_0\nu_0\kappa_0^2(t-s)}$.
	\end{lemma}
	\begin{proof}
		Define 
		\( \bw = \bv_1 - \bv_2 \) and
		\( \bpsi = \Proj{\bv_1}. \)
		Writing the evolution equation for \( \bw \) using \eqref{sys:data-assimilation},
		\[ 
			\partial_t \bw + \gamma_2 \A \bw + (\gamma_1 - \gamma_2) \A \bv_1
			+ \B{\bv_1}{\bw} + \B{\bw}{\bv_2} + \mu\Proj{\bw} = 0,
		\]
		and taking an inner-product with respect to $A^{-1}\bpsi$, we have
		\begin{multline*}
			\left<\partial_t \bw, A^{-1}\bpsi\right> + \gamma_2 \left<\bw, \bpsi\right> 
			+ (\gamma_1 - \gamma_2) |\bpsi|^2 \\
			+ \b{\bv_1}{\bw}{A^{-1}\bpsi} + \b{\bw}{\bv_2}{A^{-1}\bpsi} + \mu\left<\Proj{\bw},A^{-1}\bpsi\right> 
			= 0.
		\end{multline*}
		Using the product rule, we have 
		$\frac{d}{dt}\left<\bw, A^{-1}\bpsi\right> = \left<\partial_t \bw, A^{-1}\bpsi\right> 
			+ \left<\bw, \partial_t A^{-1}\bpsi\right>
		$
		and using the facts that $\Proj\bpsi = \bpsi$, $\Proj$ is self-adjoint, and $\Proj$ commutes with $A^{-1}$, we can rewrite the last differential inequality as
		\begin{multline*}
			\frac{d}{dt}\left<A^{-1}\Proj \bw, \bpsi\right> - \left<A^{-1}\Proj \bw, \partial_t \bpsi\right>
			+ \gamma_2 \left<\Proj \bw, \bpsi\right> + (\gamma_1 - \gamma_2) |\bpsi|^2 \\
			+ \b{\bv_1}{\bw}{A^{-1}\bpsi} + \b{\bw}{\bv_2}{A^{-1}\bpsi} + \mu\left<A^{-1}\Proj{\bw},\bpsi\right> 
			= 0.
		\end{multline*}

		Let $\bbeta := \left<A^{-1}\Proj \bw, \bpsi\right>$. Then
		\begin{multline*}
			\dot{\bbeta} + \mu\bbeta - \left<A^{-1}\Proj \bw, \partial_t \bpsi\right>
			+ \gamma_2 \left<\Proj \bw, \bpsi\right> + (\gamma_1 - \gamma_2) |\bpsi|^2 \\
			+ \b{\bv_1}{\bw}{A^{-1}\bpsi} + \b{\bw}{\bv_2}{A^{-1}\bpsi} 
			= 0,
		\end{multline*}
		so, multiplying by the integrating factor $e^{\mu (\tau - s)}$ and taking the time-integral over the interval $[s,t]$, we have
		\begin{multline}\label{eqn:nonzero-loss-equation}
			\bbeta(t) - e^{-\mu(t - s)}\bbeta(s) + (\gamma_1 - \gamma_2) \int_s^t e^{-\mu(t-\tau)}|\bpsi|^2 \:d\tau 
			\\
			= \int_s^t e^{-\mu(t-\tau)} \left(\left<A^{-1}\Proj \bw, \partial_t \bpsi\right> 
			- \gamma_2 \left<\Proj \bw, \bpsi\right> 
			\right) \;d\tau 
			\\
			-\int_s^t e^{-\mu(t-\tau)}\left( \b{\bv_1}{\bw}{A^{-1}\bpsi} + \b{\bw}{\bv_2}{A^{-1}\bpsi} 
			\right) \;d\tau
		\end{multline}

		Therefore,
		\begin{multline*}
		|\gamma_1 - \gamma_2| \frac{1}{\mu}\left( 1-e^{-\mu(t-s)}\right) \inf_{[s, t]} |\bpsi|^2 
			\leq |\bbeta(t)| + e^{-\mu(t - s)}|\bbeta(s)| \\
			+ \frac{1}{\mu}\left( 1-e^{-\mu(t-s)}\right) 
			\sup_{[s,t]}\left( 
				|A^{-1}\partial_t \bpsi|
				+\gamma_2 |\bpsi|
				\right) \|\Proj \bw\|_{C_b(H)}
			\\
			+ \frac{1}{\mu}\left( 1-e^{-\mu(t-s)}\right) 
			\sup_{[s,t]}\left( 
				(\|\bv_1\|_{L^4}+\|\bv_2\|_{L^4}) \|A^{-\frac12}\bpsi\|_{L^4}
			\right) \|\bw\|_{C_b(H)}.
		\end{multline*}
		Note that for all $t$,
		\[ |\bbeta(t)| \leq |\Proj \bw| |A^{-1} \bpsi| \leq \|\Proj \bw\|_{C_b(H)} \|A^{-1} \bpsi\|_{C_b(H)} < \infty, \]
		and letting $\delta^{N^2} = e^{-\mu(t-s)} = e^{-\mu_0\nu_0\kappa_0^2N^2(t-s)}$ (so $t - s = \frac{\ln(\frac{1}{\delta})}{\mu_0\nu_0\kappa_0^2}$), we can replace the previous inequality with
		\begin{multline*}
		|\gamma_1 - \gamma_2| \inf_{[s, t]} |\bpsi|^2
			\\
			\leq \left( 
				\mu\frac{1 + \delta^{N^2}}{1-\delta^{N^2}}\|A^{-1} \bpsi\|_{C_b(H)} + \|A^{-1}\partial_t \bpsi\|_{C_b(H)} +\gamma_2 \|\bpsi\|_{C_b(H)}
			\right) \|\Proj \bw\|_{C_b(H)}
			\\
			+ \sup_{[s,t]}\left( 
				(\|\bv_1\|_{L^4}+\|\bv_2\|_{L^4}) \|A^{-\frac12}\bpsi\|_{L^4}
			\right) \|\bw\|_{C_b(H)}.
		\end{multline*}
		
		Now, by Corollary~\ref{cor:Lipschitz_wUpperBound} with $p=0$,
		\begin{multline*}
			\|\bw\|_{C_b(H)} \leq 
			\sqrt{5} \frac{ |\gamma_1 - \gamma_2| }{ \left((\gamma_1 + \gamma_2) / 2 \right)^{1/2} }
			\frac{\|\frac12(\bv_1 + \bv_2)\|_{C_b(V)}}{\sqrt{\nu_0} \kappa_0 N}
			\\
			\leq
			\frac{\sqrt{5}}{2} \frac{ |\gamma_1 - \gamma_2| }{ \nu_0 \kappa_0 N } (\|\bv_1\|_{C_b(V)} + \|\bv_2\|_{C_b(V)}),
		\end{multline*}
		so
		\[
			|\gamma_1 - \gamma_2| \left( \inf_{[s, t]} |\bpsi|^2 - \frac{1}{N} \widetilde{M}_1 \right)
			\leq M_2 \|\Proj \bw\|_{C_b(H)}
		\]
		where \[ 
			\widetilde{M}_1 = \frac{ \sqrt{5} }{ 2 \nu_0 \kappa_0 }(\|\bv_1\|_{C_b(V)} + \|\bv_2\|_{C_b(V)})
			\sup_{[s,t]}\left( (\|\bv_1\|_{L^4}+\|\bv_2\|_{L^4}) \|A^{-\frac12}\bpsi\|_{L^4} \right)
		\]
		and \[ 
			M_2 = \mu\frac{1 + \delta^{N^2}}{1-\delta^{N^2}}\|A^{-1} \bpsi\|_{C_b(H)} + \|A^{-1}\partial_t \bpsi\|_{C_b(H)} +\gamma_2 \|\bpsi\|_{C_b(H)}.
		\]
		Finally, we use \eqref{eqn:Lady} to bound the $L^4$ norms,
		\begin{multline*}
		\sup_{[s,t]}\left( \|\bv_1\|_{L^4}+\|\bv_2\|_{L^4} \right) 
			\leq c_0 \sup_{[s,t]} \left( |\bv_1|^{\frac12}\|\bv_1\|^\frac12 
				+ |\bv_2|^{\frac12}\|\bv_2\|^\frac12 \right) \\
			\leq c_0 \left( \|\bv_1\|_{C_b(H)}^\frac12 \|\bv_1\|_{C_b(V)}^\frac12 
				+ \|\bv_2\|_{C_b(H)}^\frac12 \|\bv_2\|_{C_b(V)}^\frac12\right)
		\end{multline*}
		and (additionally using \eqref{eqn:Poincare}),
		\begin{multline*}
			\sup_{[s,t]}\left( \|A^{-\frac12}\bpsi\|_{L^4} \right) 
			\leq c_0 \|A^{-\frac12}\bpsi\|_{C_b(H)}^\frac12 \|A^{-\frac12}\bpsi\|_{C_b(V)}^\frac12
			\\ \leq c_0 \frac{1}{\kappa_0^\frac12}\|\bpsi\|_{C_b(H)}^\frac12 
				\frac{1}{\kappa_0^\frac12}\|\bpsi\|_{C_b(V)}^\frac12
			\leq \frac{c_0}{\kappa_0}\|\bv_1\|_{C_b(H)}^\frac12 
				\|\bv_1\|_{C_b(V)}^\frac12
		\end{multline*}
		so that we can replace $\widetilde{M}_1$ with $M_1$ using Lemma~\ref{lem:apriori-bounds}:
		\[ 
			\widetilde{M}_1 \leq 
			\frac{ \sqrt{5} }{ 2 \nu_0 \kappa_0 }\; 2 M_V(\bphi)
			\;2 c_0 \sqrt{M_H(\bphi)}\sqrt{M_V(\bphi)}
			\;\frac{c_0}{\kappa_0}\sqrt{M_H(\bphi)}\sqrt{M_V(\bphi)}
			= M_1.
		\]

	\end{proof}

	In the following Theorem, we apply Lemma~\ref{lem:viscosity-err-bound} to the case where $\bv_1$ is a solution of \eqref{sys:Navier--Stokes} with viscosity $\nu$. 
	\begin{theorem}\label{thm:nonzero-loss}
		Let \( \bu \in C_b(V) \) be a solution of \eqref{sys:Navier--Stokes} with viscosity 
		\( \nu \in [\nu_0, \nu_1] \),
		where \( 0 < \nu_0 < \nu_1 < \infty \) are given, and suppose that for some $n \geq 1$ and a time interval, $[s,t]$,
		\[ \inf_{[s,t]} | \Proj[n]{\bu} |^2 > 0. \]
		Choose \( \mu_0 \geq \frac{\nu_1}{\nu_0}. \) 
		If
		\begin{equation*}\label{cnd:N-nonzero-loss}
			N > \max\left\{
				\frac{8}{\nu_0\kappa_0} M_V( \bu),\: 
				n,\: 
				2\frac{M_1}{\inf_{[s,t]} | \Proj[N]{\bu} |^2}
			\right\},
		\end{equation*}
		then for any $\gamma \in [\nu_0, \nu_1]$,
		\[ |\nu - \gamma| \leq 2 \frac{M_2}{\inf_{[s,t]} | \Proj[N]{\bu} |^2} \mathcal{L}(\gamma), \]
		where $M_1$ and $M_2$ are defined as in Lemma~\ref{lem:viscosity-err-bound} with $\bv_1 = \bu$ and
		$\phi = \Proj{\bu}$.
	\end{theorem}
	\begin{proof}
		By the assumption that \( \bu\in C_b(V), \) we may take 
		\( R = \|\Proj{\bu}\|_{C_b(V)} \leq \|\bu\|_{C_b(V)} < \infty, \)
		in Theorem~\ref{thm:determining-map} and conclude that
		\( \bu = \Wmap{\nu}{\Proj{\bu}}. \) 
		The rest is an immediate consequence of Lemma~\ref{lem:viscosity-err-bound}.
	\end{proof}

	\begin{remark} 
		We chose to state Theorem~\ref{thm:nonzero-loss} with the \emph{a priori} assumption that there exists an interval of time where finitely many modes of $\bu$ are bounded away from zero. However, given the continuity of $\bu$ in time, the weaker assumption that $\bu \neq 0$ is sufficient to guarantee the existence of such a time interval.
		
		To see this, if $\bu \neq \mbf 0$. there exists $t_0$ such that $\bu (t_0) \neq \mbf 0$. 
		This implies that there exists an $\epsilon >0$ such that 
		$\inf_{t_0-\epsilon,t_0+\epsilon]} |\bu(t)| \ge \delta $. Proceeding as in the proof of Theorem \ref{thm:attrzero}, it follows that there exists $n_0$ sufficiently large such that 
		$\inf_{[t_0-\epsilon,t_0+\epsilon]} |P_{n_0}\bu(t)| \ge \frac12 \delta $. One can in fact show that if  $\mbf 0 \notin \mA_\nu$, then there exists $\delta > 0$  and $t_0 >0$  such that 
		$\inf_{[t_0,\infty)} |P_{n_0}\bu(t)| \ge  \delta $, where $n_0$ depends only on the Grashoff number and the distance  of $\mbf 0$ from the attractor.
	\end{remark} 

\section{Algorithm for Solving the Inverse Problem}\label{sec:algorithm}
	So far we have shown the validity of Problem~\ref{pbl:main-optimization} as a means of parameter recovery. We now present a method of computing its solution.
	In \cite{carlson2018}, the authors proposed the following update for the approximate viscosity, $\gamma$, as part of an algorithm to recover the true viscosity, $\nu$, using the available data only: 
	\begin{equation}\label{eqn:CHL-viscosity-update}
		\nu \approx \gamma + \frac{
			\frac12\frac1{t-s}\left(|\Proj{\bw(t)}|^2 - |\Proj{\bw(s)}|^2\right) + \mu\avg{|\Proj{\bw}|^2} 
		} {\avg{\left<\Proj{A \Wmap{\gamma}{\Proj{\bu}}}, \Proj{\bw}\right>}},
	\end{equation}
	where $\Proj{\bw} = \Proj{\bu} - \Proj{\Wmap{\gamma}{\Proj{\bu}}}$.
	Rigorous conditions for the convergence of this algorithm were obtained in \cite{Martinez_2022}.
	
	We provide a proof for the convergence of a slightly modified version of this algorithm, which is presented in Algorithm~\ref{alg:viscosity-inference}. This new algorithm is based on the approach taken in Lemma~\ref{lem:viscosity-err-bound}. In comparison to \eqref{eqn:CHL-viscosity-update}, the new update formula (see the definition of $\Gamma$ in Algorithm~\ref{alg:viscosity-inference}) is computed using inner-products with the data, $\Proj \bu$, and the time derivative of the data, as opposed to the observable error, $\Proj \bu - \Proj \bw$. Convolutions are taken with the decaying exponential factor $\tau\mapsto e^{-\mu \tau}$ over a time interval, $[s,t]$; with $\mu$ fixed, we denote this operation by $\left<\cdot\right>_s^t$,
	\[ \avg{\phi} := \frac{1}{t-s}\int_s^t e^{-\mu(t - \tau)}\phi(\tau)\;d\tau, \quad \forall \phi \in C([s,t];\mathbb{R}). \]
	\begin{algorithm} 
		\caption{Viscosity inference from data using the determining map}\label{alg:viscosity-inference}
		\begin{algorithmic}[1]
			\Require $0 < \nu_0 < \nu_1$ such that \( \nu \in [\nu_0, \nu_1] \) \Comment{bounds for the unknown viscosity, $\nu$}
			\Require $N, \mu_0$ \Comment{both chosen sufficiently large (see Theorem~\ref{thm:convergence})}
			\Require $\Proj{\bu}(t) \quad \forall t \in \mathbb{R}$ \Comment{observations of the reference solution}
			\Require $\:s < t\:$  such that $\avg{|\Proj{\bu}|^2} > 0$
			\Require $\epsilon_1 \in (0,\nu_0)$ \Comment{the convergence factor (also limited by $N$)}
			\Require $\epsilon_2 > 0$ \Comment{a stopping tolerance for the viscosity error}
			\State $\gamma \gets \gamma_0$ \Comment{an initial guess for the true viscosity}
			\Repeat
				\State $\gamma_0 \gets \gamma$
				\State $\gamma \gets \Call{$\Gamma$}{\gamma}$
			\Until $|\gamma - \gamma_0| / |\nu_1 - \nu_0| \leq \epsilon_2 / (\nu_1 - \nu_0 + \epsilon_1)$
			\Comment{the stopping condition}
			\State \Return $\gamma$
			\Procedure{$\Gamma$}{$\gamma$}
				\State $\bv \gets \Wmap{\gamma}{\Proj{\bu}}$ 
				\Comment{compute the solution of the determining map}
				\State $\bpsi \gets \Proj\bu - \Proj\bv$ 
				\Comment{compute the observable velocity error}
				\State \( c_1 \gets \frac{1}{t-s}\left(\left<A^{-1}\bpsi(t), \Proj{\bu}(t)\right> - e^{-\mu(t-s)}\left<A^{-1}\bpsi(s),\Proj{\bu}(s)\right>\right) \)
				\State \( c_2 \gets - \avg{\left<A^{-1} \bpsi, \partial_t \Proj{\bu} \right>} \)
				\State \( c_3 \gets \gamma \avg{\left<\bpsi, \Proj{\bu}\right>} \)
				\State \Return \( \gamma - \frac{c_1 + c_2 + c_3}{\avg{|\Proj{\bu}|^2}} \) 
				\Comment{return a new approximation for the viscosity}
			\EndProcedure
		\end{algorithmic}
	\end{algorithm} 
	The benefit of this new update formula is that the condition for the denominator to be nonzero is an \emph{a priori} condition depending on the observations of the reference solution only, as opposed to an \emph{a posteriori} condition which depends on the observable error, and so cannot be verified until the update is computed.

	Conditions for the convergence of Algorithm~\ref{alg:viscosity-inference} are given in the following theorem.
	\begin{theorem}\label{thm:convergence} 
		Let $\bu$ be a solution of \eqref{sys:Navier--Stokes} with viscosity $\nu \in [\nu_0, \nu_1]$,
		and suppose there exists an $n > 1$ and a time interval $[s,t]$ over which 
		\[ \avg{|\Proj[n]{\bu}|^2} > 0. \]
		Let $\epsilon_1 \in (0, \nu_0)$ and fix $\mu_0 \geq \frac{\nu_1 + \epsilon_1}{\nu_0 - \epsilon_1}.$ 
		If
		\begin{equation*}\label{cnd:N-contraction}
			N > \max\left\{
				\frac{8}{\nu_0\kappa_0} M_V( \bu),\: 
				n,\: 
				\frac{\nu_1 - \nu_0 + \epsilon_1}{\epsilon_1}\frac{M_1}{\avg{|\Proj{\bu}|^2}}
			\right\},
		\end{equation*}
		where 
		\( M_1 := \frac{ 2 \sqrt{5} c_0^2 }{ \nu_0 \kappa_0^2 }\; M_H(\bu)\,(M_V(\bu))^2, \)
		then, with $\Gamma$ as defined in Algorithm~\ref{alg:viscosity-inference}, 
		for any choice of $\gamma \in [\nu_0 - \epsilon_1, \nu_1 + \epsilon_1]$, 		
		\[ 
			|\nu -\Gamma(\gamma)| 
			< \frac{\epsilon_1}{\nu_1 - \nu_0 + \epsilon_1} |\nu - \gamma| < |\nu - \gamma|,
		\]
		and \( \Gamma(\gamma) \in [\nu_0 - \epsilon_1, \nu_1 + \epsilon_1]. \)
		Therefore, $\Gamma$ can be applied iteratively, and the result converges to $\nu$.

		Furthermore, let $\epsilon_2 > 0$ be a tolerance for stopping. We can infer closeness to the true viscosity by examining the residuals:
		\[ \frac{|\Gamma^{k+1}(\gamma) - \Gamma^k(\gamma)|}{(\nu_1 - \nu_0)} \leq \frac{\epsilon_2}{\nu_1 - \nu_0 + \epsilon_1}
		\implies |\Gamma^k(\gamma) - \nu| \leq \epsilon_2. \]
	\end{theorem}
	\begin{proof}
		Let $\bv = \Wmap{\gamma}{\Proj{\bu}}$, and let $\bw = \bu - \bv = \Wmap{\nu}{\Proj{\bu}} - \Wmap{\gamma}{\Proj{\bu}}$. 
		Proceeding as in Lemma~\ref{lem:viscosity-err-bound} up to equation \eqref{eqn:nonzero-loss-equation}, we obtain 
		\begin{multline*}
			\bbeta(t) - e^{-\mu(t - s)}\bbeta(s) + (\nu - \gamma) \avg{|\Proj{\bu}|^2}
			= \avg{\left<A^{-1}\Proj \bw, \partial_t \Proj{\bu}\right> 
			- \gamma \left<\Proj \bw, \Proj{\bu}\right> }
			\\
			-\avg{\b{\bu}{\bw}{A^{-1}\Proj{\bu}} + \b{\bw}{\bv}{A^{-1}\Proj{\bu}}} 
		\end{multline*}
		where \( \bbeta := \left<A^{-1}\Proj w, \Proj{\bu}\right>. \)
		Moving the first term on the right hand side over to the left, and dividing by the term multipying $(\nu - \gamma)$, we have
		\[
			\nu - \Gamma(\gamma) = - \frac{1}{\avg{|\Proj{\bu}|^2}} 
			\left( \avg{\b{\bu}{\bw}{A^{-1}\Proj{\bu}}} + \avg{\b{\bw}{\bv}{A^{-1}\Proj{\bu}}} \right)
		\]
		where $\Gamma(\gamma)$ is as defined in Algorithm~\ref{alg:viscosity-inference}.

		We then take the absolute value of both sides and estimate the remaining terms on the right hand side as was done in Lemma~\ref{lem:viscosity-err-bound}, and obtain
		\[ 
			|\nu - \Gamma(\gamma)| \leq \frac{1}{N} \frac{M_1}{\avg{|\Proj{\bu}|^2}} |\nu - \gamma|.
		\]

		Let $\delta > 0$. Taking $N > \frac{1}{\delta}\frac{M_1}{\avg{|\Proj{\bu}|^2}}$, we have
		\[ |\nu - \Gamma(\gamma)| \leq \delta |\nu - \gamma|. \]
		Choosing $\delta < 1$ ensures $\Gamma(\gamma)$ is closer to $\nu$ than $\gamma$, but we also need to ensure that $\Gamma(\gamma)$ satisfies the feasibility condition: $\Gamma(\gamma) \in [\nu_0 - \epsilon_1, \nu_1 + \epsilon_1]$. From the previous inequality and the feasibility condition imposed on $\gamma$, as well as the assumption $\nu \in [\nu_0, \nu_1]$, we have
		\[ \nu_0 - \delta (\nu_1 - \nu_0 + \epsilon_1) \leq \Gamma(\gamma) \leq \nu_1 + \delta (\nu_1 - \nu_0 + \epsilon_1). \] 
		Therefore, the feasibility condition for $\Gamma(\gamma)$ is satisfied with 
		\( \delta = \frac{\epsilon_1}{\nu_1 - \nu_0 + \epsilon_1}. \)
		This expression for $\delta$ is an increasing function of $\epsilon_1 \in (0,\nu_0)$, and so attains its max when $\epsilon_1  = \nu_0$. Therefore, $\delta = \frac{\epsilon_1}{\nu_1 - \nu_0 + \epsilon_1} < \frac{\nu_0}{\nu_1} < 1$, and iterating $\Gamma$ results in convergence to $\nu$:
		\[ |\Gamma^k(\gamma) - \nu| \leq \delta^k |\gamma - \nu| \to 0 \: \text{ as } \: k\to\infty. \]

		Now, for the stopping condition, note that
		\begin{multline*}
			|\Gamma^k(\gamma) - \nu|
			= |\Gamma^k(\gamma) - \Gamma^{k+1}(\gamma) + \Gamma^{k+1}(\gamma) - \nu|
			\leq |\Gamma^{k+1}(\gamma) - \Gamma^k(\gamma)| + |\Gamma^{k+1}(\gamma) - \nu|
		\end{multline*}
		so, 
		\begin{multline*}
			|\Gamma^{k+1}(\gamma) - \Gamma^k(\gamma)| 
			\geq |\Gamma^k(\gamma) - \nu| - | \Gamma^{k+1}(\gamma) - \nu | 
			\geq (1 - \delta) |\Gamma^k(\gamma) - \nu|.
		\end{multline*}
		Therefore, if after $k$ iterations we have $|\Gamma^k(\gamma) - \nu| > \epsilon_2$, then
		\[ |\Gamma^{k+1}(\gamma) - \Gamma^k(\gamma)| > (1 - \delta) \epsilon_2 = \frac{(\nu_1 - \nu_0)\epsilon_2}{\nu_1 - \nu_0 + \epsilon_1}. \]
		We get the stopping condition from the last statement by taking its contrapositive.
	\end{proof}

	In Theorem~\ref{thm:convergence}, we first chose the convergence factor $\epsilon_1$ and then chose $N$. However, we can instead let the data dictate the convergence factor; if
	\[ N > \frac{\nu_1}{\nu_0} \frac{M_1}{\avg{|\Proj{\bu}|^2}} \]
	then
	\[ \frac{(\nu_1 - \nu_0) M_1}{N \avg{|\Proj{\bu}|^2} -  M_1} < \nu_0, \]
	so we can choose $\epsilon_1$ such that
	\[ \nu_0 > \epsilon_1 \geq \frac{(\nu_1 - \nu_0) M_1}{N \avg{|\Proj{\bu}|^2} -  M_1}
		\implies |\nu -\Gamma(\gamma)| 
			< \frac{\epsilon_1}{\nu_1 - \nu_0 + \epsilon_1} |\nu - \gamma|
	.\]

	\verbose{
	\[ \frac{1}{N} \frac{M_1}{\avg{|\Proj{\bu}|^2}} \leq \frac{\epsilon_1}{\nu_1 - \nu_0 + \epsilon_1} \]
	\[ \frac{\nu_1 - \nu_0}{N} \frac{M_1}{\avg{|\Proj{\bu}|^2}} 
		+ \frac{\epsilon_1}{N} \frac{M_1}{\avg{|\Proj{\bu}|^2}} \leq \epsilon_1 \]
	\[ \frac{\frac{\nu_1 - \nu_0}{N} \frac{M_1}{\avg{|\Proj{\bu}|^2}}}{1 -  \frac{1}{N} \frac{M_1}{\avg{|\Proj{\bu}|^2}}} \leq \epsilon_1 \]
	\[ \frac{(\nu_1 - \nu_0) M_1}{N \avg{|\Proj{\bu}|^2} -  M_1} \leq \epsilon_1 \]
	\[ \frac{(\nu_1 - \nu_0) M_1}{N \avg{|\Proj{\bu}|^2} -  M_1} < \nu_0  \]
	\[ \nu_1 M_1 < \nu_0 N \avg{|\Proj{\bu}|^2} \]
	\[ N > \frac{\nu_1}{\nu_0} \frac{M_1}{\avg{|\Proj{\bu}|^2}} \]
	}

\section{Conclusions} 
	We extended the definition of the determining map to include viscosity as a variable. We then studied the inverse problem of determining the viscosity from the velocity data using the determining map, and found conditions for the well-posedness of the inverse problem, as well as conditions for the regularity of the inverse mapping, and defined an iterative algorithm for obtaining its solution. In future work, we plan to compare this algorithm with gradient based solvers applied to Problem~\ref{pbl:main-optimization}, leveraging the recent sensitivity results in \cite{Carlson_Larios_2021} for the analysis.

	One consequence of our results is that, on the attractor, viscosity is determined, just like all higher modes, from the low modes only. Therefore, we have extended the concept of determining modes to include a parameter of the equation as being determined from finitely many low modes. This also has consequences for how different attractors (indexed by the viscosity) can intersect, as no two solutions on the attractor can overlap for any interval of time without the viscosities and solutions being identical
    
	For simplicity of presentation, we chose to frame our results as well as Algorithm~\ref{alg:viscosity-inference} in terms of solutions on the attractor, which requires having data for the reference solution for all time and computing the solution of the determining map for all time. However, our results easily extend to the initial value problem, where the data are only required on a bounded interval of time, and the determining map is solved as an initial value problem with an arbitrary initial condition. Algorithm~\ref{alg:viscosity-inference} can also be modified as an ``on-the-fly'' algorithm, as was done in \cite{carlson2018}.


\begin{appendix} 
	\section{Proof of Ladyzhenskaya's inequality}
	For completeness, we provide a proof of Ladyzhenskaya's inequality, valid on the torus in 2D. 
	\begin{equation}
		\|\bphi\|_{L^4(\Omega)} \leq |\bphi|^{\frac12} \left(\tfrac{1}{L} |\bphi| + \|\bphi\|\right)^{\frac12}
	\end{equation}
	\begin{proof}
	Let \( \bphi \in C([0,L];\mathbb{R}) \) be periodic (i.e. \(\bphi(L) = \bphi(0)\) and $\bphi$ can be extended to $\mathbb{R}$ via the relation 
	\( \bphi(x + nL) := \bphi(x)\: \forall x\in[0,L], n\in \mathbb{N}\)). 
	Denote the mean of $\bphi$ by 
	\(\overline{\bphi} = \frac{1}{L} \int_0^L \bphi(x) dx.\)
	Then by the Mean Value Theorem, there is a point $x_0 \in [0,L]$ such that 
	\( \bphi(x_0) = \overline{\bphi}. \)	

	For any \(x \in [0,L]\), we have 
	\[ 	
		\frac{1}{2}\bphi^2(x) = \frac12 \overline{\bphi}^2 + \int_{x_0}^x \bphi(z) \bphi'(z) dz 
		\leq \frac12 \overline{\bphi}^2 + \int_{x_0}^x |\bphi(z)| |\bphi'(z)| dz, 
	\]
	and by periodicity,
	\[ 
		\frac{1}{2}\bphi^2(x) = \frac12 \overline{\bphi}^2 - \int_{x}^{x_0 + L} \bphi(z) \bphi'(z) dz 
		\leq \frac12 \overline{\bphi}^2 + \int_{x}^{x_0 + L} |\bphi(z)| |\bphi'(z)| dz, 
	\]
	so
	\[ 
		\bphi^2(x) \leq \overline{\bphi}^2 + \int_{x_0}^{x_0+L} |\bphi(z)| |\bphi'(z)| dz 
		= \overline{\bphi}^2 + \int_{0}^{L} |\bphi(z)| |\bphi'(z)| dz. 
	\]

	This easily extends to vector valued \( \bphi \in C([0,L];R^2): \) 
	\[
		|\bphi(x)|^2 = \bphi_1^2(x) + \bphi_2^2(x)
		\leq \overline{\bphi_1}^2 + \int_{0}^{L} |\bphi_1(z)| |\bphi_1'(z)| dz 
		+ \overline{\bphi_2}^2 + \int_{0}^{L} |\bphi_2(z)| |\bphi_2'(z)| dz,
	\]
	which we can simplify by applying \eqref{eqn:Jensen},
	\begin{multline*}
		\overline{\bphi_1}^2 + \overline{\bphi_2}^2 
		= \left( \frac{1}{L} \int_0^L \bphi_1(x) dx \right)^2 + \left(\frac{1}{L} \int_0^L \bphi_2(x) dx \right)^2 \\
		\leq \frac{1}{L} \int_0^L \bphi_1^2(x) dx + \frac{1}{L} \int_0^L \bphi_2^2(x) dx
		= \frac{1}{L} \int_0^L |\bphi(x)|^2 dx 
	\end{multline*}
	and Cauchy-Schwarz,
	\[ 
		|\bphi_1(z)| |\bphi_1'(z)|+ |\bphi_2(z)| |\bphi_2'(z)|
		= \begin{bmatrix} |\bphi_1(z)| \\ |\bphi_2(z)| \end{bmatrix} 
			\cdot \begin{bmatrix} |\bphi_1'(z)| \\ |\bphi_2'(z)| \end{bmatrix}
		\leq |\bphi(z)| |\bphi'(z)|
	\]
	to obtain
	\[
		|\bphi(x)|^2 \leq \frac{1}{L} \int_0^L |\bphi(z)|^2 dz
		+ \int_{0}^{L} |\bphi(z)| |\bphi'(z)| dz.
	\]
	
	Now, if \( \bphi \in C(\Omega; \mathbb{R}^2) \), 
	using the estimates for a one-dimensional domain, for any \( (x,y) \in \Omega, \) we have
	\begin{multline*}
		|\bphi(x,y)|^4 = |\bphi(x,y)|^2 |\bphi(x,y)|^2 \\
		\leq \left( \frac{1}{L} \int_0^L |\bphi(z, y)|^2 dz + \int_{0}^{L} |\bphi(z,y)| |\partial_1 \bphi(z,y)| dz \right)\\
		\left( \frac{1}{L} \int_0^L |\bphi(x, z)|^2 dz + \int_{0}^{L} |\bphi(x,z)| |\partial_2 \bphi(x,z)| dz \right),
	\end{multline*}
	so
	\begin{multline*}
		\|\bphi\|_{L^4(\Omega)}^4 = \int_0^L\int_0^L|\bphi(x,y)|^4\,dx\,dy \\
		\leq \int_0^L\left( \frac{1}{L} \int_0^L |\bphi(z, y)|^2 dz + \int_{0}^{L} |\bphi(z,y)| |\partial_1 \bphi(z,y)| dz \right)dy\\
		\int_0^L\left( \frac{1}{L} \int_0^L |\bphi(x, z)|^2 dz + \int_{0}^{L} |\bphi(x,z)| |\partial_2 \bphi(x,z)| dz \right)dx\\
		\leq \left( \frac{1}{L} |\bphi|^2 + |\bphi| |\partial_1 \bphi| \right)
		\left( \frac{1}{L} |\bphi|^2 + |\bphi| |\partial_2 \bphi| \right) \\
		\leq \left( \frac{1}{L} |\bphi|^2 + |\bphi| \|\bphi\| \right)
		\left( \frac{1}{L} |\bphi|^2 + |\bphi| \|\bphi\| \right),
	\end{multline*}
	where the second to last line follows from the Cauchy-Schwarz inequality.
	\end{proof}

\end{appendix}

\end{document}